\documentclass{amsart}

\usepackage{amssymb, 
     amsmath, 
     amscd,
     thmtools,
     mathtools
     }

\usepackage{listings}
\usepackage{enumitem}
\usepackage{tcolorbox}

\usepackage{xcolor}
\usepackage{hyperref}
\usepackage[nameinlink]{cleveref}

\newcommand{\CC}{{\mathbb C}}

\newcommand{\PP}{{\mathbb P}}
\newcommand{\kk}{\mathbb{C}}

\newcommand{\Gm}{\mathbb{G}_{\mathrm m}}
\newcommand{\im}{\operatorname{im}}
\newcommand{\GL}{\operatorname{GL}}
\newcommand{\SL}{\operatorname{SL}}
\newcommand{\id}{\operatorname{id}}
\newcommand{\Hom}{\operatorname{Hom}}
\newcommand{\Mat}{\operatorname{Mat}}
\newcommand{\QHM}{Q_{\mathrm{HM}}}
\newcommand{\QHMs}{Q_{\mathrm{HM},s}}

\renewcommand{\i}{\textbf{i}}
\DeclareMathOperator{\Gr}{Gr}
\DeclareMathOperator{\sing}{sing}
\DeclareMathOperator{\sm}{sm}
\renewcommand{\SS}{\mathcal{S}}
\DeclareMathOperator{\SetC}{\mathcal{C}}
\renewcommand{\AA}{\mathbb{A}}
\DeclareMathOperator{\dom}{dom}
\newcommand{\set}[2]{\left\{\,#1 \ | \ #2\,\right\}}
\newcommand{\Bigset}[2]{\left\{\,#1 \ \Big| \ #2\,\right\}}
\newcommand{\cquot}[2]{#1\mkern-2mu\mathord{/\mkern-6mu/}\mkern-3mu#2}

\theoremstyle{plain}
\newtheorem{thm}{Theorem}[section]
\newtheorem{lem}[thm]{Lemma}
\newtheorem{cor}[thm]{Corollary}
\newtheorem{prop}[thm]{Proposition}



\theoremstyle{definition}
\declaretheorem[sibling=thm,name=Definition,qed={$\diamondsuit$}]{definition}
\declaretheorem[sibling=thm,name=Example,qed={$\diamondsuit$}]{ex}
\declaretheorem[sibling=thm,name=Remark,qed={$\diamondsuit$}]{re}

\Crefname{re}{Remark}{Remarks}
\crefname{ex}{example}{examples}
\Crefname{ex}{Example}{Examples}
\crefname{thm}{theorem}{theorems}
\Crefname{thm}{Theorem}{Theorems}
\crefname{lem}{lemma}{lemmas}
\Crefname{lem}{Lemma}{Lemmas}
\crefname{prop}{proposition}{propositions}
\Crefname{prop}{Proposition}{Propositions}
\crefname{cor}{corollary}{corollaries}
\Crefname{cor}{Corollary}{Corollaries}
\crefname{definition}{definition}{definitions}
\Crefname{definition}{Definition}{Definitions}
\crefname{remark}{remark}{remarks}
\Crefname{remark}{Remark}{Remarks}
\crefname{example}{example}{examples}
\Crefname{example}{Example}{Examples}
\crefname{examplex}{example}{examples}
\Crefname{examplex}{Example}{Examples}
\AddToHook{cmd/appendix/before}{%
    \crefalias{section}{appendix}%
    \crefalias{subsection}{appendix}
    \crefalias{Section}{Appendix}%
    \crefalias{Subsection}{Appendix}
}
\crefname{subsection}{subsection}{subsections}
\Crefname{subsection}{Subsection}{Subsections}

\begin{document}

\title{Symmetric subrank and its border analogue}

\begin{abstract}
    The symmetric subrank of homogeneous polynomial is the largest number of terms in a {\em diagonal} form to which it can be specialized by a (typically non-invertible) linear variable substitution. Building on earlier work by Derksen-Makam-Zuiddam and Biaggi-Chang-Draisma-Rupniewski for ordinary tensors, we determine the asymptotic behavior of symmetric subrank and symmetric border subrank of degree-$d$ forms as the number of variables tends to infinity. Furthermore, by using results from geometric invariant theory we show that for cubic (resp. quartic) 
    forms the symmetric subrank and symmetric border subrank coincide if the latter is at most 
    three (resp. two).
\end{abstract}

\author{Benjamin Biaggi}
\address{Mathematical Institute, University of Bern, Alpeneggstrasse 22, 3012 Bern, Switzerland}
\email{benjamin.biaggi@unibe.ch}

\author{Jan Draisma}
\address{Mathematical Institute, University of Bern, Sidlerstrasse 5, 3012 Bern, Switzerland}
\email{jan.draisma@unibe.ch}

\author{Koen de Nooij}
\address{Finalix Business Consulting, Beethovenstrasse 48, 8002 Zürich, Switzerland}
\email{koen@denooij.com}

\author{Immanuel van Santen}
\address{Mathematical Institute, University of Bern, Alpeneggstrasse 22, 3012 Bern, Switzerland}
\email{immanuel.van.santen@unibe.ch}

\thanks{The authors were supported by project grant 10000940 from the Swiss National Science Foundation.}


\maketitle

\setcounter{tocdepth}{1}
\tableofcontents

\section{Introduction and main results}
Throughout, we let $V$ be an $n$-dimensional vector space over an algebraically closed field $K$. Let $d$ be a nonnegative integer. 
We denote by $S^dV $ the \emph{$d$-symmetric power of $V$}, the quotient space of $V ^{ \otimes d}$ by the vector subspace spanned by all differences $T - \sigma (T)$, where $T  \in V^{\otimes d}$ and $\sigma \in S_d$ is a permutation of $\{1, \dots, d\}$.
The elements of $S^d V$ are called symmetric tensors in this note, and we use the notation $v_1 v_2 \cdots v_d$ for the image in $S^d V$ of $v_1 \otimes \cdots \otimes v_d$; these symmetric tensors are called {\em simple}.
The action of a linear map $\varphi : V \to W$ on symmetric tensors in $S^dV$ is determined by $\varphi (v_1 \cdots v_d ):= (\varphi v_1) \cdots (\varphi v_d)$.
If $V = W$ and $ \varphi \in \GL (V)$, this yields the usual left action of $\GL (V) $ on $S^d V$ and we sometimes denote $\varphi \cdot f$ to emphasize this fact. Given two tensors $f \in S^d V$ and $h \in S^d W$ and a linear map $\varphi \in  \Hom (V,W)$ with $\varphi f = h$, we say that $h$ is a \emph{symmetric restriction} of $f$, denoted by $h \leq _s f$.

\subsection{Symmetric subrank}
For reasons that will become clear soon (\Cref{prop:Char}), we assume throughout that 
\begin{center}
   \begin{tcolorbox}[
  colback=white,
  colframe=black,
  arc=3mm,
  width = \textwidth
   ]
   \vspace{-0.2cm}
   \begin{align} 
   \text{$d$ is not a power of the characteristic exponent of $K$ times $0,1,$ or $2$.} \tag{*} \label{charass}
   \end{align}
   \end{tcolorbox}   
\end{center}

\begin{definition}
    Let $f \in S^d V$ be a symmetric tensor. The \emph{symmetric subrank} of $f$ is defined as 
    \[
    Q_s(f) := \max \{r \mid \exists \text{ linear map } \varphi: V \to K^r; \, \varphi  f  = I_r \},
    \]
    where $I_r := e_1^d + \dots + e_r^d $ is the \emph{$r$-th unit tensor} and $e_1 \dots , e_r$ is the standard basis of $K^r$.
\end{definition}

The assumption \eqref{charass} implies that, for $r \geq 2$, $I_r$ is not the $d$-th power of an element of $V$.

The symmetric subrank was introduced in \cite{CFTZ22symsubrank}. Note that $d=2$ is not allowed by our condition \eqref{charass}, but otherwise, for a quadratic form over a quadratically closed field of characteristic different from $2$, 
the symmetric subrank would equal the rank of its Gram matrix. % (this is because in characteristic $\neq 2$, every symmetric bilinear form has an orthogonal basis. (see e.g.~\cite[Ch. XV, Theorem~3.1]{La2002Algebra}) that can be normalized when the field is quadratically closed.) 
%\cite[Lemma~2.5 and Theorem~2.11]{CFTZ22symsubrank}.

Symmetric subrank is dual to the symmetric tensor rank  (also called Waring rank for polynomials). Given a symmetric tensor $f \in S^d V$, the symmetric tensor rank equals 
\[
R_s (f) = \min  \{r \mid \exists \text{ linear map } \psi: K^r \to V; \, \psi  I_r  = f \}.
\]
Symmetric subrank is a symmetric analogue of Strassen's subrank for ordinary tensors  \cite{Strassen1987}, defined as follows.
If $T\in V_1 \otimes \cdots \otimes V_d$  is a tensor, the \emph{subrank} $Q(T)$ is the maximal $r$ for which there
exist linear maps $\varphi_i:V_i \to K^r,\ i=1,\dots,d$, such that 
\[
(\varphi_1 \otimes \cdots \otimes \varphi_d) T = e_1^{\otimes d} + \dots + e_r^{\otimes d}.
\]
For $d=3$, $T$ can be thought of as a bilinear map $V_1^* \times V_2^* \to V_3$, and $Q(T)$ is thought of as the {\em value} of $T$, which says how many independent scalar multiplications can be linearly embedded in $T$.

\subsection*{Different perspectives on the symmetric subrank}
If $K$ has characteristic $0$, then symmetric tensors in $S^d V$ are in one-to-one correspondence with tensors in $V^{\otimes d}$ fixed by all permutations in $S_d$. Let $T \in V^{\otimes d}$ be a symmetric tensor in this latter sense. Then the symmetric subrank is
\[
Q_s(T) = \max\{ r \mid \exists \text{ linear map } \varphi: V \to K^r; \, \varphi ^ {\otimes d} T = \sum_{i = 1}^r e_i ^{\otimes d}\},
\]
and it is immediate that $Q(T) \geq Q_s(T)$.
The subrank analogue of Comon's conjecture is disproved in \cite{shitov22_comon_for_subrank}: there exists a symmetric tensor $T$ over the complex numbers with $Q(T)>Q_s(T)$.

In this setting, a symmetric tensor of order three $T  \in V\otimes V \otimes V$ in particular defines a symmetric bilinear form $T : V^{\ast} \times V^\ast \to V$.  The symmetric subrank encodes the ``symmetric value'' of this form, i.e. the number of linearly independent scalar multiplication which can be symmetrically embedded into $T$: If the linear map $\varphi : V \to K^r$ satisfies $\varphi ^ {\otimes 3} T = I_r$, we have $ \varphi (T (\varphi ^t (a), \varphi ^t (b))) = I_r (a,b) $ for all $a,b \in K^r $, where $\varphi^t:K^r \to V^*$ is the map dual to $\varphi$ and 
\[
I_r : K^r \times K^r \to K^r;\, ((a_1 \dots , a_r),(b_1, \dots , b_r)) \mapsto  (a_1b_1,  \dots , a_r b_r). 
\]
In some settings, it is more convenient to work with homogeneous polynomials of degree $d$, which are elements of $S^d ((K^{n})^ \ast )$, and with a more geometric description of the subrank. For a homogeneous polynomial $f \in K[x_0, \dots , x_{n-1}]_d$, the symmetric subrank of $f$ is the largest number of terms in a diagonal form to which it can be specialized by a linear variable substitution
\[
Q_s(f) = \max \{r \mid \exists \text{ linear map } \varphi :K^r \to K^n; \; 
f \circ \varphi = x_0^d + \dots + x_{r-1}^d\}.
\]
Given a homogeneous polynomial $f$, the zero set $\mathcal{V}(f)$ defines a projective hypersurface in $\PP ^{n-1}$. The subrank of $f$ is the maximal integer $r$ such that we can intersect the hypersurface $\mathcal{V}(f)$ with $n-r$ hyperplanes and the resulting hypersurface inside $\PP ^{r-1}$ is, up to the action of $\GL _r$, defined by the Fermat form $x_0^d + \dots + x_{r-1}^d = I_r$:
\[
Q_s (f) = \max \{r \mid \exists \text{ hypersurfaces }H_i; \; \mathcal{V}(I_r) \in \GL_r \cdot (\mathcal{V}(f) \cap H_1 \cap \dots \cap H_{n-r}) \}.
\]
We will use this geometric description in 
\Cref{sec:cubic}.
\subsection{Generic (symmetric) subrank}
The locus of tensors in $V_1 \otimes \cdots \otimes V_d$ with a given subrank $r$ is constructible; therefore there is a unique integer $r$ for which this locus is dense. We call that $r$  the \emph{generic subrank} of tensors in  $V_1 \otimes \cdots \otimes V_d$. Setting $n_i := \dim V_i$, we denote this generic subrank by $Q(n_1, \dots , n_d)$. In \cite[Theorem~3.7]{PSS-exactvaluesgenericsubrank}, it is shown that for $d\geq 3$ and every algebraically closed field
\[
Q(n_1, \dots , n_d) = \min \{ n_1, \dots , n_d , \lfloor (n_1 + \dots + n_d - (d-1))^{1/(d-1} \rfloor\}.  
\] 
We point out that the assumption that $K$ is  algebraically closed does not appear explicitly in \cite{PSS-exactvaluesgenericsubrank}. But it is needed---indeed, \cite{PSS-exactvaluesgenericsubrank} builds on the earlier paper \cite{DMZ24} that does need this requirement. In \Cref{sec:gen_sym_subrank} we follow an approach similar to \cite{DMZ24} to show:
\begin{thm}
\label{thm:generic_symsubrank}
Let $\dim V = n$ (and assume \eqref{charass}). Then the generic symmetric subrank $Q_s(n)$ of elements in $S^d V$ satisfies 
\[
\lim_{n \to \infty} \frac{Q_s(n)}{(d! \cdot n)^{1/(d-1)}} = 1.
\]
\end{thm}
\begin{re}
    For $d=3$, the denominator above equals $\sqrt{6n}$, which is larger than the value $Q(n,n,n)=\lfloor \sqrt{3n-2} \rfloor$ from \cite[Corollary 2.17]{PSS-exactvaluesgenericsubrank}. At first, this seems to contradict the fact that $Q(T) \geq Q_s(T)$ for symmetric tensors in characteristic zero. However, the point here is that a sufficiently general tensor in $V \otimes V \otimes V$ is {\em not} symmetric.
\end{re}
We will show that the generic symmetric subrank is well-defined and give an upper bound to it using similar arguments to those in \cite{DMZ24}. In \Cref{sec:gen_sym_subrank_lower_bound} we will use work by Hochster-Laksov \cite[Theorem 1]{Hochster87} on the linear syzygies of sufficiently general forms to give a lower bound that matches the upper bound up to lower-order terms.

\subsection{(Generic) symmetric border subrank} 
The \emph{border subrank} of a tensor in $V_1 \otimes \dots \otimes V_d$ is the maximal $r$ for which $I_r$
lies in the closure of the set
\[ \{(\varphi_1 \otimes \cdots \otimes \varphi_d)T \mid 
\varphi_i \in \Hom(V_i,K^r), i=1,\dots,d\}. \] 

\begin{definition}
    Let $f \in S^d V$ be a symmetric tensor. The \emph{symmetric border subrank} of $f$ is defined as 
    \[
    \underline{Q}_s(f) := \max \{r \mid I_r \in \overline{ \{ \varphi  f \mid \varphi \in \Hom (V, K^r) \}} \}.
    \]   
\end{definition}
It follows from the definitions that $1 \leq Q_s(f) \leq \underline{Q}_s(f) \leq n$ for all nonzero $f \in S^dV$.  In \cite{BCDR25-bordersubrank}, it is proven that the generic border subrank has the same growth rate as the generic subrank, see \S 1.3 an \S 1.4 of \cite{BCDR25-bordersubrank} for the result. In \Cref{sec:gen_sym_border_subrank}, we follow these ideas to show that the same is true for the generic symmetric border subrank.

\begin{thm} \label{thm:generic_symmetric_border_subrank}
Let $\dim V = n$ (and assume \eqref{charass}). The generic symmetric border subrank of $S^d V$ is  $\mathcal{O}(n^{1/(d-1)})$ for $n \to \infty.
$
\end{thm}

\subsection{Comparing symmetric subrank and symmetric border subrank}
In \Cref{sec:difference}, we prove several results relating symmetric subrank and symmetric border subrank. We recall that the generic border subrank of order $3$ tensors is, for large enough $n$, strictly larger than the generic subrank, see \cite[Theorem~4]{BCDR25-bordersubrank}. We were not able to show a similar result in our setting. The issue is the following: there is a natural notion of {\em Hilbert-Mumford subrank} $\QHM(T)$ of a tensor $T$ that satisfies $Q(T) \leq \QHM(T) \leq \underline{Q}(T)$, and \cite[Proof of Theorem~4]{BCDR25-bordersubrank} shows that in fact the first 
inequality is strict for sufficiently general order-$3$ tensors $T$ when $n$ is large enough. This is not true for its symmetric analogue (we define $\QHMs(f)$ in \Cref{sec:difference}):

\begin{thm} \label{thm:HMsubrank}
    For any $f \in S^d V$, the symmetric Hilbert-Mumford subrank $\QHMs(f)$ equals the symmetric subrank $Q_s(f)$.
\end{thm}

Nevertheless, we do give an example of a quintic $f$ with 
$Q_s(f) < \underline{Q}_s(f)$, see \Cref{ex:Quintic}. 
Finally, \Cref{sec:cubic} is concerned with $d=3, 4$, 
i.e., with cubic and quartic forms.
For low values of the symmetric border subrank, we show that $Q_s$ and $\underline{Q}_s$ coincide on cubics and quartics. Specifically:
\begin{thm} \label{thm:LowRank}
Assume that $K=\CC$. Suppose that 
$f \in S^3 V$ and 
$\underline{Q}_s(f) \leq 3$ or $f \in S^4 V$ and $\underline{Q}_s(f) \leq 2$. Then in fact $Q_s(f)=\underline{Q}_s(f)$.
\end{thm}
The theorem above is a consequence of the following much stronger statement. 

\begin{thm} \label{thm:Cubics}
Let $K = \CC$, $(d, k) \in \{(3, 2),(4, 1)\}$, $n \geq k$, 
and let $X \subset \PP^{n+1}$ be a 
hypersurface defined by a homogeneous polynomial $f \in \CC[x_0,\ldots,x_{n+1}]_d$
of degree $d$ such that $X$ is normal when $d = 3$ and $X$ is reduced when $d = 4$. 
Assume that $f$ uses at least $k+2$ variables in the sense that it cannot be put into $\CC[x_0,\ldots,x_k]_d$ by means of an invertible linear coordinate transformation. 
Then the set
\[
\Bigset{ X \cap P}{
\begin{array}{l}
\text{$P \subset \PP^{n+1}$ a projective $k$-dimensional} \\
\text{linear subspace not contained in $X$}
\end{array}}
\]
contains representatives of all isomorphism classes of smooth degree-$d$ hypersurfaces in $\PP^k$.
\end{thm}

The condition that $f$ uses at least $k+2$ variables is needed as otherwise $X \cap P$ is isomorphic to a fixed degree-$d$ hypersurface in $\PP^k$ for all choices of $P$ such that $X \cap P$ is smooth.
The normality of $X$ in the case $(d, k) = (3, 2)$ 
is needed, since otherwise its singular locus has codimension $1$, and then $X \cap P$ is always singular; by an analogous argument the reducedness of $X$ is used in case 
$(d, k) = (4, 1)$. Details are provided in \Cref{sec:cubic}, see in particular 
\Cref{re:LX}, \Cref{Lem.L_X_non-empty} and \Cref{Lem.Pi_non-empty}.

The choice of the pairs $(3, 2)$ and $(4, 1)$ for $(d, k)$ comes from the fact that
in these cases the moduli space of smooth degree-$d$ hypersurfaces in $\PP^k$ 
is particularly small, namely isomorphic to the affine line 
$\AA^1$, see \Cref{Subsection.mpduli}.

\section{Properties of the unit tensor}

In this short section, we establish a few fundamental properties of the unit tensor $I_r=e_1^d+ \dots + e_r^d \in S^d K^r$. These properties are well known when the characteristic of $K$ is zero or $>d$, but require a little extra care under our weaker assumption \eqref{charass}. Recall that the {\em slice rank} of $f \in S^d V$ is the minimal codimension of a linear subspace of $\PP V^*$ on which $f$ vanishes identically; see, e.g. \cite{Ballico23}.

\begin{prop} \label{prop:Char}
Under our standing assumption \eqref{charass}, the unit tensor $I_r$ has the following properties:
    \begin{enumerate}
    \item $I_r$ is not contained in $S^d U$ for any proper subspace $U$ of $K^r$;
    \item the slice rank of $I_r$ equals $\lceil r/2 \rceil$; 
    \item the orbit $\GL_r \cdot I_r \subseteq S^d K^r$ has dimension $r^2$.
    \end{enumerate}
\end{prop}

\begin{proof}
    Write $d=d' \cdot p^k$ where $p$ is the characteristic exponent of $K$, $k$ is a nonnegative integer, and $d'$ is an integer that maps to a nonzero element of $K$ and that by \eqref{charass} is $>2$. Now $I_r=(I_r')^{p^k}$ where $I_r':=e_1^{d'} + \cdots + e_r^{d'}$. The fact that $d'$ is nonzero in $K$ implies that the hypersurface $Z$ in $\PP (K^r)^*$ defined by $I_r'$ is smooth. 

    For (1), let $U \subseteq K^r$ be a subspace such that $I_r \in S^d U$. Then also $I_r' \in S^{d'} U$. Now if $U$ is strictly smaller than $K^r$, then $Z$ is a cone, and the fact that its degree $d'$ is $>1$ implies that $Z$ is singular, a contradiction. 

    In (2), $\leq$ follows from the fact that for all $i$ with $2i \leq r$ the binary form $e_{2i-1}^d + e_{2i}^d$ has a linear factor, and so does the form $e_{r}^d$ if $r$ is odd. Now $I_r$ vanishes identically on the linear space defined by these $\lceil r/2 \rceil$ linear forms. For $\geq$ assume that $I_r$ vanishes identically on the space defined by linear forms $\ell_1,\ldots,\ell_s \in K^r$. Then so does $I_r'$, and it follows that $I_r'=\ell_1 h_1 + \dots + \ell_s h_s$ for suitable $h_i \in S^{d'-1} K^r$. Using the Leibniz rule, we see that every partial derivative of $I_r'$ lies in the ideal generated by the $\ell_i$ and the $h_i$ together, and so the singular locus of $Z$ has codimension at most $2s$. Since $Z$ is smooth, we find that $2s \geq r$. (This argument is essentially due to Birch, and also shows that the {\em strength} of $I_r$ is equal to $\lceil r/2 \rceil$, see \cite[page 3]{Lampert24}.)

    For (3), it suffices to show that the orbit $\GL_r \cdot I_r'$ has dimension $r^2$ (this orbit is homeomorphic to $\GL_r \cdot I_r$ via the map $f \mapsto f^{(p^k)}$). Since $\dim(\GL_r)=r^2$, this follows when the Lie algebra of the stabiliser in $\GL_r$ of $I_r'$ is $\{0\}$. For every $A$ in the Lie algebra of the stabiliser, the coefficient of $\epsilon$ in 
    \[ (e_1+\epsilon Ae_1)^{d'} + \dots + (e_r+\epsilon Ae_r)^{d'}\]
    equals zero. This coefficient equals 
    \[ d' e_1^{d'-1} (Ae_1) + \dots + d' e_r^{d'-1} (Ae_r). \]
    Now $d'$ is nonzero in $K$, and since $d'>2$ the linear subspaces $e_1^{d'-1} \cdot K^r,\ldots,e_r^{d'-1} \cdot K^r \subseteq S^{d'}K^r$ are linearly independent. Hence for said coefficient to be zero, we need that $Ae_i=0$ for all $i$, so that $A=0$, as desired.
\end{proof}

We remark that only for (3) we use that $d'>2$; for (1) and (2) it suffices that $d'>1$.

\section{Generic symmetric subrank}
\label{sec:gen_sym_subrank}
The main goal of this section is to proof  \Cref{thm:generic_symsubrank}. This is done in 3 steps: Showing that there exists a generic symmetric subrank $Q_s(n)$ and then giving an upper bound and a lower bound on $Q_s(n)$.
We start by establishing that the generic symmetric subrank is well-defined.

Let $u_1, \dots , u_r$ be linearly independent vectors in $V$, and set $U := \langle u_1,\ldots,u_r \rangle_K$. We identify $I_r$ with $u_1^d + \dots + u_r^d$. We fix another subspace $W$ satisfying $V = U   \oplus W$, so that $S^d V = \bigoplus _{e = 0}^d S^e U \cdot S^{d-e}W$. Given $f \in S^d V$, we define $f_U$ as the component of $f$ in the summand $S^d U$. Define the affine subspace
\[
X_r := \{f  \in S^d V \mid f_U = I_r\}.
\]

Let $\Phi _r $ be the map defined by applying linear transformations on elements of $X_r$,
\begin{align*}
\Phi _r : \GL (V)  \times X_r &\to S^d V \\ 
(g,f) & \mapsto g \cdot f .
\end{align*} 

Let $\mathcal{C}_{\geq r} \subseteq S^d V$ be the set of homogeneous polynomials with symmetric subrank at least $r$:
\[
\mathcal{C}_{\geq r} := \{f \in S^d V \mid Q_s(f) \geq r\}.
\]
The following arguments are symmetric analogues of \cite[\S2.1]{DMZ24}.
\begin{lem}
\label{lm:image_phi_r}
The image of $\Phi _r$ equals $\mathcal{C}_{\geq r} $.
\end{lem}
\begin{proof}
We have $X_r \subseteq \mathcal{C}_{\geq r}$ since for any $f \in X_r$ and for $\pi_U:V \to U$ the projection with kernel $W$ we have $\pi_U f=I_r$. Now $\im(\Phi_r) \subseteq \mathcal{C}_{\geq r}$ follows from the fact that $Q_s$ is invariant under $\GL(V)$.

For the opposite inclusion $\mathcal{C}_{\geq r}  \subseteq \im (\Phi  _r) $, let $f \in \mathcal{C}_{\geq r} $. Then there exists a linear map $\varphi :V \to U$ with $\varphi f= I_r$. Since $\varphi f \in S^d(\im(\varphi))$, part (1) of \Cref{prop:Char} implies that $\varphi$ is surjective, so we can write $\varphi$ as $\pi_U \circ \psi$, where $\psi \in \GL(V)$. We then have $h:=\psi \cdot f \in X_r$, and hence $f=\psi^{-1}\cdot h \in \im (\Phi _r)$.
\end{proof}

\begin{prop}
Given $d$ and an $n$-dimensional vector space $V$, there exists a unique $r$, such that for all $f \in  Y$, where $Y$ is a Zariski-dense open subset of $S^d V$, we have $Q_s(f) = r.$ We call this $r$ the \emph{generic symmetric subrank} of symmetric tensors in $S^d V$ and denote it by $Q_s(n)$.
\end{prop}
\begin{proof}
By Chevalley's theorem and the lemma above, the set $\mathcal{C}_i  = \mathcal{C}_{\geq i} \setminus \mathcal{C}_{\geq i+1}$ is constructible, and we have $S^d V = \bigsqcup _i \mathcal{C}_i$. As $S^d V$ is irreducible and the $\mathcal{C}_i$ are constructible, precisely one $\mathcal{C}_i$ is dense, and this contains a dense open subset $Y$ of $S^dV$.
\end{proof}

\subsection{An upper bound on \texorpdfstring{$Q_s (n)$}{the generic symmetric subrank}} The upper bound uses the same idea as in \cite{DMZ24}, namely bound 
$\dim(\mathcal{C}_{\geq r})$ from above using the fiber dimension theorem for the map $\Phi _r$.
\begin{lem}
The dimension of $\mathcal{C}_{\geq r} $ is bounded from above by
\[
\dim (\mathcal{C} _{\geq r}) \leq \binom{n+d-1}{d} - \binom{r+d-1}{d} + nr.
\]
\end{lem}
\begin{proof}
The fiber dimension theorem gives 
\[ \dim  (\mathcal{C}_{\geq r}) = \dim (\GL (V) \times X_r) - \dim (\text{general fiber of }\Phi _r). \]
The dimension of $X_r$ equals $\dim(S^d V)-\dim(S^d U)$, so that 
\[
\dim (\GL (V) \times X_r) = n^2 +  \binom{n+d-1}{d} - 
\binom{r+d-1}{d}.
\]
We now bound the fiber dimension from below. Given an $f \in \mathcal{C}_{\geq r}$, we have $f = g \cdot h$ for $g \in \GL (V) $ and $h \in X_r$. The varieties $\Phi _r^{-1} (h) $ and $\Phi _r^{-1} (f) $ are isomorphic, as they differ by the action of the invertible element $g$. So it suffices to bound $\dim (\Phi _r^{-1} (h))$ from below. We define an $n(n-r)$-dimensional subgroup of $\GL (V) $ by
\[
L := \{ g \in \GL (V) \mid g(W)=W \text{ and $g$ induces the identity on $V/W$} \} .
\]
For all $g \in L$, we have $(g^{-1}, g \cdot h)  \in \Phi _r^{-1} (h)$, so that this fiber has dimension at least $n(n-r)$. Subtracting this from the expression for $\dim(\GL(V) \times X_r)$ above proves the lemma.
\end{proof}

To show the upper bound, we also need the following lemma from \cite{DMZ24}. The proof is identical to the one given there.
\begin{lem}[Symmetric version of {\cite[Lemma 2.3]{DMZ24}}]
\label{lm:dimension_of_Cr}
The generic symmetric subrank equals
\begin{flalign*}
&&Q_s(n) = \max \left\{ r \mid \dim (\mathcal{C}_{\geq r}) = \dim (S^d V)  \right\}. && \square \end{flalign*}
\end{lem}
\begin{prop}
\label{prop:upper_bound}
Let $n = \dim (V)$. The generic symmetric subrank of $S^dV$ satisfies 
\[
Q_s(n) \leq (d! \cdot n)^{1/(d-1)}.
\]
\end{prop}
\begin{proof}
Using both lemmas from above, for $r = Q_s(n)$, we have 
\[ \dim (\mathcal{C}_{\geq r} ) = \dim (S^d V) = \binom{n+d-1}{d} \] and also 
$\dim (\mathcal{C}_{\geq r}) \leq \binom{n+d-1}{d} - \binom{r+d-1}{d} + nr$. We combine this and find 
\begin{align*}
\binom{n+d-1}{d} & \leq \binom{n+d-1}{d} - \binom{r+d-1}{d} + nr \\
\Leftrightarrow \; \binom{r+d-1}{d} & \leq nr.
\end{align*}
Using the inequality $r^d/d! \leq \binom{r+d-1}{d}$, we obtain $r^{d-1}/d! \leq n$ and hence the desired inequality.
\end{proof}
For order 3 tensors, we get a more precise upper bound.
\begin{prop}
Let $\dim (V) = n$. The generic symmetric subrank of $S^3 V$ satisfies
\[
Q_s(n) \leq \lfloor \sqrt{6n + 1/4}-3/2 \rfloor.
\]
\end{prop}
\begin{proof}
As in the proof of \Cref{prop:upper_bound}, for $r = Q_s(n)$, we get 
\begin{align*}
& {n+2 \choose 3} \leq {n+2 \choose 3} - {r+2 \choose 3} + nr \\
\Leftrightarrow \;&{r+2 \choose 3} \leq nr \\
\Leftrightarrow \;& \frac{(r+2)(r+1)r}{6}  \leq nr \\
\Leftrightarrow \;& r \leq \sqrt{6n + 1/4}-3/2.
\qedhere
\end{align*}
\end{proof}

\begin{re}
    It would be interesting to see if the methods from \cite{PSS-exactvaluesgenericsubrank} can be adapted to show that the upper bound in the previous proposition is the exact value of $Q_s(n)$.
\end{re}

\begin{re}
Given a symmetric tensor $f\in S^d V$ with $\dim (V) = n$, we can look at the growth rate in $n$ for fixed $d$, but also in $d$ for fixed $n$. We know that the generic {\em Waring rank} equals $ \lceil \frac{1}{n} {n+d -1 \choose d} \rceil$, at least in characteristic zero and except in a small number of cases \cite{Alexander_Hirschowitz}. This grows for fixed $n$ in $d$. On the other hand, the symmetric {\em subrank} for a nonzero tensor lies in $\{1,\ldots,n\}$, so the generic symmetric subrank for fixed $n$ is bounded. The inequality ${r+d-1 \choose d} \leq nr$ shows that for fixed $n$ and $d \geq 2n $, the generic symmetric subrank of $S^d (K^{n})$ equals 1.
\end{re}
\subsection{A lower bound on \texorpdfstring{$Q_s(n)$}{the generic symmetric subrank}}
\label{sec:gen_sym_subrank_lower_bound}
The following result is a combination of well known facts about the rank of differentials and that in a similar way as in \cite[Lemma 3.3]{DMZ24}, one can show that $\im ((d\Phi _r)_{(g,f)}) = g \cdot \im ((d\Phi _r)_{(\id,f)}) $. We forgo the proof, as their argument for the subrank also works for the symmetric subrank.

\begin{cor}[Symmetric version of {\cite[Corollary 3.4]{DMZ24}}]
If for some $f \in X_r$ we have 
\[
\im ((d\Phi _r)_{(\id,f)})  = S^dV,
\]
then $Q_s(n) \geq r$. Conversely, if the characteristic of $K$ equals $0$ and 
$Q_s(n) \geq r$, then the set of $f$ with the property above is open and dense in the affine space $X_r$. 
\end{cor}
The space $\Hom (V,V) = \mathfrak{gl} (V)$ is the tangent space of $\GL(V)$ at $\id \in \GL(V)$ and
\[
Y_r := \{f \in S^d V\mid f_U=0\} = \bigoplus_{e=0}^{d-1} S^e U \cdot S^{d-e} W 
\] is the tangent space of $X_r$ at any point ($X_r$ is an affine space).

The contraction map 
\begin{align*}
    V^\ast \times S^d V &\to S^{d-1}V\\
    (y, f ) &\mapsto f(y, \cdot ),
\end{align*} is determined by $f(y,\cdot )= \sum _i y(v_i) v_1 \cdots v_{i-1}v_{i+1}\cdots v_d$ for simple $f = v_1\cdots v_d$.  
\begin{prop} \label{prop:Derivative}
The differential of $\Phi _r$ at the point $(\id,f) \in \GL (V) \times X_r$ is the linear map 
\[
(d \Phi _r) _{(\id,f)} : \Hom (V,V)\times Y_r  \to S^dV, 
\]
determined by 
\[
(y \otimes u,h) \mapsto  f(y ,\cdot ) u   +  h 
\]
for $y \otimes u \in V^\ast \otimes V = \Hom (V,V) .$
\end{prop}
\begin{proof}
    We compute the coefficient of $\epsilon$ in $\Phi_r (\id + A \epsilon, f + h \epsilon)$, where $A\in \Hom (V,V)$ and $h \in Y_r $ are elements in the tangent spaces of $\GL (V) $ and $X_r$. By the linearity of the differential and the linearity of $\Phi _r$ in the second argument, we can restrict the computations to the simple elements $A =  y \otimes u$ and $ f= v_1 \cdots v_d.$ We compute modulo $\epsilon ^2$
    \begin{align*}
        \Phi_r &(\id + (y \otimes u) \epsilon, v_1 \cdots v_d + h \epsilon) = \\ &  (v_1 \cdots v_d) + y(v_1 ) v_2 \cdots v_d u \epsilon  + \dots + v_1 \cdots v_{d-1} y(v_d) u \epsilon + h \epsilon,
    \end{align*}    
    which gives the desired differential.
\end{proof}

The differential is a linear map, so we have 
\begin{align*}
\im ((d\Phi _r)_{(\id,f)}) &=  (d\Phi _r)_{(\id,f)}(\Hom (V,V) \times \{0\} ) + (d\Phi _r)_{(\id,f)}(\{0\} \times Y_r )\\
& = (d\Phi _r)_{(\id,f)}(\Hom (V,V) \times \{0\} ) + Y_r.
\end{align*}
Therefore, to ensure that the differential is surjective, we need to show that the space $(d\Phi _r)_{(\id,f)}(\Hom (V,V) \times \{0\} )$ projects surjectively onto $S^d U$. We will now show that, if $r$ is not too large, then we can choose $f$ so that $(d\Phi _r)_{(\id,f)}(\Hom (W,U) \times\{0\} )$ {\em equals} $S^dU$;  here $\Hom(W,U)$ is regarded as a subspace of $\Hom(V,V)$ by extending a linear map $\varphi:W \to U$ with zero on $U$.
\begin{cor}
\label{cor:Q_s(n)_geq_r_if_span_equals_SdU}
Let $h_1,\dots ,h_{n-r}$ be elements of $S^{d-1} U$. If $Uh_1 + \dots + Uh_{n-r} = S^d U$, then  $Q_s(n) \geq r$.
\end{cor}
\begin{proof}
    Let $w_1,\ldots,w_{n-r}$ be a basis of $W$, $y_1,\ldots,y_{n-r}$ the dual basis of $W^*$, extended by zero on $U$ to linear functions on all of $V$, and set $f:=w_1 h_1 + \dots + w_{n-r} h_{n-r}$. For $u \in U$ and $i \in \{1,\ldots,n-r\}$, \Cref{prop:Derivative} yields
    \[ (d\Phi_r)_{(\id,f)}(y_i \otimes u,0)=uh_i.\]
    Since the symmetric tensors $uh_i$ span $S^d U$ as $u$ and $i$ vary, we have 
    \[ (d\Phi_r)_{(\id,f)} (\Hom(W,U) \times \{0\}) = S^d U,\]
    as desired.
\end{proof}
We now use a result from \cite{Hochster87} to find a lower bound on the generic 
symmetric subrank.
\begin{prop}
\label{prop:lower_bound}
The generic symmetric subrank of $S^dV$ satisfies 
\[
1 \leq \liminf_{n \to \infty} \frac{Q_s(n)}{(d! n)^{1/(d-1)}}.
\]
\end{prop} 
\begin{proof}
If $(n-r)r \geq \binom{r+d-1}{d}$, then for sufficiently 
general $h_1, \ldots, h_{n-r} \in S^{d-1}U$ we have 
$\sum_{i=1}^{n-r} U h_i = S^d U$ by
\cite[Theorem~1]{Hochster87} and a hence $Q_s(n) \geq r$ by
\Cref{cor:Q_s(n)_geq_r_if_span_equals_SdU}.
Thus we must ensure that
\begin{equation}
\label{eq:alexander_hirschowitz_condition_for_r}
(n-r)r \geq  {r+d-1 \choose d}, \text{ which is equivalent to } n-r \geq \frac{1}{d}\binom{r+d-1}{d-1}.
\end{equation}
Using the inequality 
\[
{r+d-1 \choose d-1} \leq \frac{(r+d-1)^{d-1}}{(d-1)!}
\]
we see that it suffices that 
\[ n \geq \frac{(r+d-1)^{d-1}}{d!} + r.\]
Let $r(n)$ be the largest solution of $\frac{(x+ d-1)^{d-1}}{d! n} + \frac{x}{n} = 1$.
Then 
\[
    q(n) \coloneqq \frac{r(n) + d-1}{(d! n)^{1/(d-1)}}
\]
is the largest solution of $x^{d-1} + a_1(n) x - a_0(n) = 0$, where 
$a_1(n) = \frac{d!}{ (d! n)^{(d-2)/(d-1)}}$ and $a_0(n) = -(1 + \frac{d-1}{n})$.
As $x^{d-1} + a_1(n) x - a_0(n) = 0$ has exactly one solution $\geq 0$,
we conclude that $q(n)$ tends to a root of
$x^{d-1} -1 = 0$. Hence,  
\[
    1 = \lim_{n \to \infty} q(n) = 
    \lim_{n \to \infty} \frac{r(n)}{(d! n)^{1/(d-1)}} ,
\]
which shows the result.
%}
%and hence that 
%\[ r \leq (d! n)^{1/(d-1)} - o(n^{1/(d-1)}) \quad \text{ for } n \to \infty. %\qedhere\]
\end{proof}
For $d=3$ we can compute a more precise bound:
\begin{prop}
The generic symmetric subrank of $S^3V$ satisfies 
\[
Q_s(n) \geq \lfloor \sqrt{6n+73/4}-9/2 \rfloor
\]
\end{prop}
\begin{proof}
We recall that we have to show that \Cref{eq:alexander_hirschowitz_condition_for_r} is satisfied.
\begin{align*}
& (n-r)r \geq {r+2 \choose 3}\\
\Leftrightarrow \; & n-r \geq \frac{1}{6} r^2 + \frac{1}{2}r + \frac{1}{3} \\
\Leftrightarrow \;&r \leq \sqrt{6n+73/4}-9/2. \qedhere
\end{align*}
\end{proof}

\begin{proof}[Proof of \Cref{thm:generic_symsubrank}.]
The asymptotic behaviour of $Q_s(n)$ follows from the upper bound in \Cref{prop:upper_bound} and the lower bound in \Cref{prop:lower_bound}.
\end{proof}

\section{Generic symmetric border subrank}
\label{sec:gen_sym_border_subrank}
The goal of this section is to give an upper bound on the generic symmetric border subrank. The growth rate of this upper bound will be the same as for the symmetric subrank, but unlike in that case we do not know the precise multiplicative constant. The proof are adapted versions of the proof in \cite{BCDR25-bordersubrank}.

First, we need the following result to see that the generic symmetric border subrank is well-defined for all spaces of symmetric tensors.
\begin{prop}[Symmetric version of {\cite[Proposition 1]{BCDR25-bordersubrank}}]
For $n,r,V$, the set $\mathcal{D}_r \subseteq S^dV$ of tensors of border subrank precisely $r$ is a constructible set. Therefore, the same holds for the sets $\mathcal{D}_{<r}$ and $\mathcal{D}_{\geq r}$. \hfill $\square$
\end{prop}
Let $f \in S^dV$ be a tensor with symmetric border subrank $\geq r$. Then $I_r \in \overline{\GL (V)  \cdot S^d V}$, where we recall that $I_r$ can be taken as $u_1^d + \dots + u_r^d$ for any linearly independent collection of vectors $u_1,\ldots,u_r \in V$. 

Using {\cite[Proposition 6]{BCDR25-bordersubrank}}, there exists a tensor $h \in \GL (V) \cdot I_r$ and a one-parameter subgroup 
$\lambda : \Gm \to \GL (V)$ such that 
\begin{equation}
\label{eq:lim_1-PSG}
\lim _{t \to 0} \lambda (t) \cdot f = \lim _{t \to \infty} \lambda (t) \cdot  h = :h_0,
\end{equation}
where both limits exist. We recall some facts about one-parameter subgroups. For every one-parameter subgroup $\lambda : \Gm \to \GL (V)$, there is a unique set of {\em weights}
 $b_1 < \dots < b_N$, with $N \leq {n + d -1 \choose d}$, such that $S^dV = \bigoplus _{i=1}^N (S^d V)_{b_i}$, where 
\[
(S^d V)_b :=\{v \in S^dV \mid \forall t \in \Gm: \lambda(t) \cdot v  =  t^b v\} \neq \{0\}.
\]
We define $(S^dV)_{>0} = \bigoplus _{a >0 } (S^d V)_a$ and similarly for $(S^dV)_{<0},(S^dV)_{\geq 0},(S^dV)_{\leq 0},$ and $(S^dV)_{0}$. By \Cref{eq:lim_1-PSG}, $f \in (S^dV)_{\geq 0}$, $h \in (S^dV)_{\leq 0}$ and their components $f_0, h_0$ in $ (S^dV)_{0}$ are equal.

Next let $v_1, \dots , v_n$ be an eigenbasis of $V$ with respect to the action of $\lambda(t)$ on $V$, i.e., $\lambda (t) v_i = t^{a_i} v_i$ for all $t$, and so $a_1 , \dots , a_n$ are the weights of the action of $\lambda $ on $V$. This yields
\[
\lambda (t) \cdot v_{i_1} \cdots v_{i_d} =  t^{a_{i_1} + \dots + a_{i_d}} v_{i_1} \cdots v_{i_d}  ,
\]
for $i_1,\ldots,i_d \in \{1,\ldots,n\}$. So the elements $v_{i_1} \cdots v_{i_d} $, with ${i_1}  \leq \dots \leq {i_d}$  are an eigenbasis for $S^dV$ and the sums $a_{i_1} + \dots + a_{i_d} $ are the weights of $\lambda$ on $S^dV$. 

For a fixed ordered tuple of integers $(a_1 \leq \dots \leq a_n)$, we define the incidence variety with these given weights 
\begin{align*}
Y:= \{(\lambda , h, f) \mid &\lambda \text{ has weights } a_1,\ldots,a_n \text{ on } V,\\ &h \in \GL(V) \cdot I_r \text{ and } \lim _{t \to 0} \lambda (t) \cdot f = \lim _{t \to \infty} \lambda (t) \cdot h \}.
\end{align*}
The set $\mathcal{D}_{\geq r}$ of tensors with border subrank at least $r$ is contained in the countable union, over all tuples $(a_1 \leq \dots \leq a_n)$ of weights on $V$, of the image of the projection of $Y$ to the third component. Therefore the dimension of $\mathcal{D}_{\geq r}$ is the maximal dimension of those projections; see \cite[beginning of Section 5]{BCDR25-bordersubrank} why this is even true if the field is countable. 

The proofs of the following two lemmas use the same ideas as in the non-symmetric setting.

\begin{lem}[Symmetric version of {\cite[Lemma~10]{BCDR25-bordersubrank}}]
\label{lm:S0_sym}
The number of parameters in $K$ needed to determine $f_{>0}$, i.e., the dimension of the image of the map $Y \to S^dV$ defined by 
\[
(\lambda , h, f) \mapsto f_{>0}, \text{the component of $f$ in $(S^dV)_{>0}$}
\]
is at most 
\[
{n+d-1 \choose d} - {s+d-1 \choose d} + 2s(n-s)
,\]
where $s:=\lceil r/2 \rceil.$
\end{lem}
\begin{proof}
Consider $(\lambda,h,f) \in Y$ and let $v_{1},\dots,v_{n} \in V$ be an
eigenbasis of
$\lambda$, so that
\[ \lambda (t) \cdot v_{i}= t^{a_{i}} v_{i},\quad t \in \Gm. \]
Now $(S^dV)_{\leq 0}$ is the space spanned by all tensors 
\begin{equation} 
\label{eq:leqzero}
v_{i_1} \cdots  v_{i_d}  \text{ with } i_1 \leq \dots \leq i_d \text{ and }
a_{i_1} + \dots + a_{i_d}\leq 0. 
\end{equation}
We claim that $a_s \leq 0$. Indeed, assume for a contradiction that $a_{s} >0$. Then any simple tensor $v_{i_1} \cdots v_{i_d} \in (S^d V)_{\leq 0}$ is divisible by one of $v_1, \dots , v_{s-1}$ and so symmetric tensors in $(S^d V)_{\leq 0}$ have slice rank $\leq s-1$. This contradicts the fact that $h \in (S^dV)_{\leq 0}$, lying in the orbit of $I_r$, has slice rank precisely $s$ by item (2) of \Cref{prop:Char}.

Let $W \subseteq V$ be the subspace spanned by
$v_{s+1},\dots,v_{n}$,
so that $V=U \oplus W$.
Using $a_s \leq 0$, we find
\[ f_{>0} \in (S^dV)_{>0} \subseteq 
W \cdot S^{d-1}V. \]
The right-hand side is a space of dimension $\dim S^d V - \dim S^d U = {n+d-1 \choose d} - {s + d -1 \choose d}$. Furthermore, $U$ (respectively, $W$) is a point in the Grassmannian
of $s$-dimensional (respectively, $(n-s)$-dimensional) subspaces in
$V$. Both of these Grassmannians have dimension $s(n-s)$. Adding
these dimensions to the upper bound on the dimension of $(S^d V)_{>0}$ completes the proof.
\end{proof}

\begin{lem}[Symmetric version of {\cite[Lemma~11]{BCDR25-bordersubrank}}]
\label{lm:Tg0_sym}
The number of parameters in $K$ needed to determine $h_{0}$, i.e., the dimension of the image of the map $Y \to S^dV$ defined by 
\[
(\lambda , h, f) \mapsto h_{0}, \text{the component of $h$ in $(S^dV)_{0}$}
\]
is at most 
\[
r(n-r) + r^2 = rn
.\]
\end{lem}

\begin{proof}
Consider $(\lambda,h,f) \in Y$, so that $\lim_{t \to \infty}
\lambda(t) \cdot h=h_0$ exists. Define
\[ Q=Q(\lambda):=\{g \in \GL (V) \mid \lim_{t \to \infty} 
\lambda(t)g \lambda(t)^{-1} \text{ exists in $\GL (V)$} \}. \]
This is a parabolic subgroup of $\GL (V)$ (see 
\cite[Proposition~8.4.5]{Sp1998Linear})
that clearly contains the centralizer of $\im(\lambda)$ and in particular $\im(\lambda)$ itself. Furthermore, for any $g \in Q$ we have
\[ (g \lambda(t) g^{-1}) \cdot h = 
g \cdot (\lambda(t) g^{-1} \lambda(t)^{-1}) \cdot (\lambda(t) \cdot h) \to (g
g_0) \cdot h_0, \quad (t \to \infty) \]
for some $g_0 \in \GL (V)$.

Fix a basis $e_{1},\ldots,e_{n}$ of $V$ such that 
\[ h=I_r=\sum_{i=1}^r e_{i}^d, \]
and let $B_n$ be the Borel subgroup of $\GL_n$ consisting of upper triangular
matrices relative to this basis. Any two parabolic subgroups
intersect in at least a maximal torus (see, e.g., \cite[Corollary
14.13]{Borel91}), hence some maximal torus
$D$ of $\GL_n$ is contained in $Q \cap B_n$. Any two maximal tori in $Q$
are conjugate (see, e.g., \cite[Corollary 11.3]{Borel91}), and therefore there exists a $g \in Q$ such that $\mu:=g
\lambda g^{-1}$ maps $\Gm$ into $D \subseteq B_n$. By the previous paragraph,
$h_0':=\lim_{t \to \infty} \mu(t) \cdot h$ lies in the $\GL_n$-orbit of $h_0$.

Now let $U$ be the subspace of $V$ spanned by
$e_{1},\dots,e_{r}$. Because the image of $\mu$ consists of upper-triangular matrices, 
the space $S^dU$ is preserved by $\mu$. Hence $h_0'$ also
lies in this space, and it is evidently contained in the orbit closure
of the unit tensor $I_r$ under $\GL (U)$. 

Since $h_0$ lies in the $\GL_n$-orbit of $h_0'$, there also exists an
$r$-dimensional subspace $W
\subseteq V$ such that $h_0$ lies in $S^dW$
and is contained in the orbit closure of a unit tensor in $S^d W$. The subspace $W$ is a point in a Grassmannian of dimension
$r(n-r)$. In addition, we need to add the dimension of the orbit closure $\overline{\GL (W) \cdot I_r}$ of the unit tensor in $S^dW$, which equals $r^2$ by item(3) of \Cref{prop:Char}. 
Thus, as desired, $h_0$ is determined by at most  
\[ r (n-r) + r^2 = rn \]
parameters. 
\end{proof}

\begin{prop}[Symmetric version of {\cite[Theorem 3]{BCDR25-bordersubrank}}] \label{prop:BoundDr}
The locus $\mathcal{D}_{\geq r}$ of tensors with symmetric border subrank $\geq r$ has dimension at most 
\[
{n+d-1 \choose d} - {s+d-1 \choose d}  + 2s(n-s)+rn,
\]
where $s=\lceil r/2 \rceil$.
\end{prop}
\begin{proof}
It suffices to show that the image of the morphism $\pi: Y \to S^dV,\ 
(\lambda,h,f) \mapsto f$ has at most the stated dimension. This morphism factorizes via the map
\[ Y \to S^dV \times S^dV,\quad (\lambda,h,f) \mapsto (h_0,f_{>0}) \]
and the addition map $S^dV \times S^dV \to S^dV$. So an upper bound on the
dimension of $\im(\pi)$ is given by adding the dimension bounds from
\Cref{lm:S0_sym} and \Cref{lm:Tg0_sym}.
\end{proof}

\begin{proof}[Proof of \Cref{thm:generic_symmetric_border_subrank}]
If $r$ is the generic border subrank, then the dimension of $\mathcal{D}_{\geq r}$ equals that of $S^dV$. By \Cref{prop:BoundDr} we have 
\begin{align*}
& {n+d-1 \choose d} \leq {n+d-1 \choose d} - {s+d-1 \choose d}  + 2s(n-s) + rn \\
\Rightarrow \; & { s+d-1 \choose d}  \leq 2s(n-s)+2sn=2s(2n-s)
\end{align*}
We find that $s^d / d! \leq 4sn$, so $s \leq (4d! \cdot n)^{1/(d-1)}$, and hence $r \leq 2 \cdot (4d! \cdot n)^{1/(d-1)}$.
\end{proof}

\begin{re}
    We note that the coefficient of $n^{1/(d-1)}$ above is a factor $2 \cdot 4^{1/(d-1)}$ larger than the coefficient in the upper bound of \Cref{prop:upper_bound}. This is not a large gap; this is related to our discussion in \Cref{sec:difference}.
\end{re}

\section{Difference between symmetric subrank and symmetric border subrank} \label{sec:difference}
In this section, we want to analyze differences between the symmetric subrank and symmetric border subrank. In \cite{BCDR25-bordersubrank}, it was shown that for $n$ large enough, the generic border subrank is always strictly larger than the generic subrank. We were not able to do a similar construction in the symmetric case, and we will see below why, indeed, the method used there cannot possibly work for the symmetric subrank.
\begin{definition}
\label{Def.HMs}
Let $f \in S^d V$ be a symmetric tensor. We define the \emph{symmetric Hilbert-Mumford subrank}  as
\[
\QHMs(f) := \max \{r \mid \exists \text{ homomorphism } \lambda : \Gm \to \GL (V) : \lim _{t \to 0 } \lambda (t) \cdot f \in \GL (V) \cdot I_r \}.
\]
\end{definition} 
It is clear that for all symmetric tensors $f$, the symmetric Hilbert-Mumford subrank lies in between the symmetric subrank and symmetric border subrank:
\[
Q_s(f) \leq \QHMs(f)  \leq \underline{Q}_s(f).
\]

\subsection{Symmetric subrank equals symmetric Hilbert-Mumford subrank}
The constructions in \cite{Chang24} show that for subrank, the Hilbert-Mumford subrank can be strictly larger than the subrank: Chang showed that the locus of order-3 tensors of maximal (that is, equal to $n$) Hilbert-Mumford subrank has dimension $\Theta (n^3)$ for $n \to \infty$, but the locus of order-$3$ tensors with maximal subrank equals the orbit of $I_n$, whose dimension is $\Theta (n^2)$ for $n \to \infty$. Furthermore, in \cite{BCDR25-bordersubrank} it was shown that the Hilbert-Mumford subrank is strictly greater than the subrank for sufficiently general tensors. We will now show that in the symmetric setting, this does not happen.
\begin{lem}
\label{lm:max_sym_bordersubrank}
Let $f \in S^dV$ be a tensor with maximal symmetric border subrank, i.e., $I_n \in \overline{\GL (V) \cdot f}$, then there exists $g \in \GL(V)$ with $g \cdot f = I_n$.
\end{lem}
\begin{proof}
If $I_n \in \overline{\GL (V) \cdot f}$, then we have $\GL (V) \cdot I_n \in \overline{\GL (V) \cdot f}$. For the algebraic group $\GL (V)$ we have $\overline{\GL (V) \cdot f} = \GL (V) \cdot f \sqcup Z$, where $\dim (Z ) < \dim (\overline{\GL (V) \cdot f})$ and $\GL (V) \cdot f$ is open \cite[Proposition in {\MakeUppercase{\romannumeral 1}}.1.8]{Borel91}. It holds that $\dim (\GL (V) \cdot I_n )= n^2$ by item (3) of \Cref{prop:Char}, which is already the maximal dimension of an orbit under $\GL (V)$, and hence equals $\dim (\GL (V)\cdot f)$. This implies that $I_n \in \GL (V) \cdot f$.
\end{proof}

\begin{proof}[Proof of \Cref{thm:HMsubrank}.]
Let $f \in S^d V $ with $\QHMs (f)\geq r$. 
We want to show that $Q_s(f) \geq r$. There is a one-parameter subgroup $\lambda : \Gm \to \GL (V) $ such that 
\[
\lim _{t \to 0} \lambda (t) \cdot f = \sum _{i = 1}^r v_i ^d \in \GL (V) \cdot I_r,
\]
where $v_1,\ldots,v_r$ are linearly independent. Let $e_1,\dots,e_n$ in $V$ be the weight vectors of $\lambda$, with respective weights $a_1,\dots ,a_n$ (not necessarily ordered). Without loss of generality, in the $r \times n$-coefficient matrix of the $v_i$ relative to the $e_j$, the left-most $r\times r$ block has nonzero determinant. Let $U$ be the span of $e_1,\dots ,e_r$ and let $\pi$ be the projection $V\to U$ with kernel spanned by $e_{r+1},\dots ,e_n$ that restricts
to the identity on $U$.

The two maps $\pi $ and $\lambda $ commute: for all $i=1,\ldots,n$ we have
\[
\pi (\lambda (t)(e_i)) = \pi (t^{a_i} e_i ) = t^{a_i} \pi (e_i ) = \lambda (t) (\pi (e_i)),
\]
where the last equality holds because $\pi(e_i)$ is either $e_i$ or zero.
We want to show that $\pi (f) \in S^d U$ also has symmetric Hilbert-Mumford subrank $r$. It holds
\[
\lambda (t) \cdot \pi (f) = \pi \cdot \lambda (t) f \xrightarrow{t \to 0} \pi \cdot \sum _{i = 1}^r v_i^d = \sum _{i = 1}^r \pi (v_i ) ^d. 
\]
We restrict $\lambda $ to $S^d U$ to see that $\pi (f) $ has maximal Hilbert-Mumford subrank, which implies using \Cref{lm:max_sym_bordersubrank} that $\pi (f) $ also has symmetric subrank $r$ and so $Q_s(f) \geq r$ as desired.
\end{proof}

\subsection{An example where \texorpdfstring{$\underline{Q}_s(f) > Q_s (f)$}{symmetric border subrank is greater than symmetric subrank}}

Using Groebner bases, one can theoretically compute the symmetric subrank for any given symmetric tensor $f$: Fix a basis $e_1, \dots, e_n$ of $V$. The subrank is the maximal $r$ where the equation $A f = e_1^d + \cdots + e_r^d $ for $A \in \Mat (r\times n)$ has a solution.

If we want to find examples $f \in S^d(K^n)$ with $\underline{Q}_s(f) > Q_s (f)$, we can show that the subrank equals $r$ using the computations explained above and then either find 
$g \colon K \setminus \{0\} \to \GL_n$ with $\lim _{t \to 0} g(t) f = I_{r+1}$ or show that the differential of the map $\Phi _f: \Mat ((r+1) \times n ) \to S^d (K^{r+1})$, 
$A \mapsto A f$ is full dimensional at a sufficiently general matrix $A$. The second method can be done effectively using a computer, but is not a necessary condition for $\underline{Q}_s (f) \geq r$.

Using this method, we were not able to find an example of an order $3$ or $4$ tensor. The Groebner basis computations did not finish anymore for $n >4$. However, for $d=5$, the following example has strictly larger symmetric border subrank. This example is inspired by Shitovs example of a symmetric tensor with symmetric subrank strictly smaller than subrank \cite{shitov22_comon_for_subrank}.

\begin{ex} \label{ex:Quintic}
Let $f= u^2 w^3 + v^5 + uvw^3 \in S^5 \CC ^3 $, where $u,v,w$ are basis vectors of $\CC^3$. Then there is no solution to 
\[
\begin{pmatrix}
    a_{11} & a_{12} & a_{13}\\
    a_{21} & a_{12} & a_{23}
\end{pmatrix} f = u^5 + v^5 = I_2
\]
using a Groebner basis computation.
So $Q_s(f) <2$ and therefor $Q_s(f) = 1$. 

Let 
\[ g(t) = 
\begin{pmatrix}
    t^3  &  0  &  t^{-3}\\
    0 & 1 & 0 \\
    0 & 0 & t^2
\end{pmatrix},
\]
then it holds 
\[
g(t) f = (t^3u + t^{-3}w)^2  t^6w^3 + v^5 + (t^3 u+ t^{-3}w)  v  t^6w^3 \to w^5 + v^5.
\]
This implies that $\underline{Q}_s(f) \geq 2$ and by \Cref{lm:max_sym_bordersubrank}  $\underline{Q}_s(f) < 3$, so  $\underline{Q}_s(f) = 2.$
\end{ex}

Note that if $f \in S^d \CC^3$ for $d \leq 4$, then we always have $\underline{Q}_s(f) = Q_s(f)$ (\Cref{lm:max_sym_bordersubrank} and \Cref{thm:LowRank}).

\section{Linear sections of cubic and quartic hypersurfaces} 
\label{sec:cubic}
In the remainder of this paper, we work over the field $\kk$ of complex numbers.
Let $n \geq 1$ be an integer and
let $X \subseteq \PP^{n+1}$ be a hypersurface of degree $d > 2$, i.e.\
$X$ is given in $\PP^{n+1}$ by one non-zero homogeneous polynomial 
$f \in \kk[x_0, \dots, x_{n+1}]_d$ of degree $d$,
%\beni{I would add $x_{n_1}]_3$ if we later use subscript 3}\immu{I agree}
where $x_0, \dots, x_{n+1}$ denote the homogeneous coordinates of $\PP^{n+1}$.
We denote by $X_{\sing}$ the singular locus of $X$ and by $X_{\sm} \coloneqq X \setminus X_{\sing}$
the smooth locus of $X$. Moreover, for any $x \in X$ we denote by $T_x X \subseteq \PP^{n+1}$
the projective tangent space of $x$ at $X$.

We fix $k \in \{1, \ldots, n\}$ and denote by $\Gr = \Gr(k+1, n+2)$ the Grassmannian of all 
(projectively $k$-dimensional) linear spaces in $\PP^{n+1}$. 
The intersection $X \cap P$ always denotes the scheme-theoretic intersection.
The goal of this section is to prove (a slightly more precise version of) \Cref{thm:Cubics} and to derive \Cref{thm:LowRank} from it. Hence, we are mostly interested in the cases where
$(d, k)$ is $(3, 2)$ or $(4, 1)$.

Let $\Pi_{0} \subseteq \Gr$ be the subset of those 
$P \in \Pi$ such that $P \not\subseteq X$ and $X \cap P$ is smooth.
We start with some general observations.

\begin{lem}
   \label{Lem.X_cap_P_singular}
   Let $P \in \Gr$ with $P \not\subseteq X$ and let $x \in X \cap P$. 
   Then $X \cap P$ is singular at $x$ if and only if 
   $P$ is contained in $T_x X$. 

   In particular, $\Gr \setminus \Pi_{0}$ contains every $P \in \Gr$ with 
   $P \subseteq T_x X$ for some $x \in X \cap P$.
\end{lem}

\begin{proof}
   Assume first that $P \in \Gr$ is contained in $T_x X$. 
   After a linear coordinate change
   we may assume $x = [1:0: \cdots : 0]$ and $P$ is given by the vanishing of 
   $x_{k+1}, \dots, x_{n+1}$. Since $P \subseteq T_x X$ we get
   \[
      f(1, x_1, \dots, x_{n+1}) = f_1(x_{k+1}, \dots, x_{n+1}) + f_2(x_1, \dots, x_{n+1}) + \ldots + f_d(x_1, \dots, x_{n+1}) \, ,
   \]
   where $f_i$ is homogeneous of degree $i$.
   Then $P \cap X$ is given in the open chart $\{ x_0 \neq 0 \}$ of $P$ by the vanishing of 
   $f_2(x_1, \ldots, x_k, 0, \dots, 0) + \ldots + f_d(x_1, \ldots, x_k, 0, \dots, 0)$. Since $P \not \subseteq T_x X$, this is a nonzero polynomial all of whose terms have degrees $>1$. Hence $P \cap X$ is singular at $x = [1:0:\cdots:0]$.

   Now, assume that $X \cap P$ is singular at $x$. Then
   $\dim T_x(X \cap P) = \dim T_x P$, as $X \cap P$ is given by one equation in 
   $P \simeq \PP^k$. Since $T_x P = P$, we get that $P \subseteq T_x X$.
\end{proof}

\subsection{Characterization of the emptyness of $\Pi_0$}

\begin{lem}
   \label{Lem.Pi_non-empty}
   The following statements are equivalent:
   \begin{enumerate}
      \item \label{Lem.Pi_non-empty.Assertion1} 
      $\dim X_{\sing} + k \leq n$;
      \item \label{Lem.Pi_non-empty.Assertion2} 
      The set $\Pi_{0}$ is non-empty.
   \end{enumerate}
\end{lem}

\begin{proof}
   Assume first that $\dim X_{\sing} + k \leq n$. If $k = 1$, then 
   $X$ is reduced and for a general line $P \in \Gr$, the intersection $P \cap X$
   contains exactly $d$ points, i.e.\ $X \cap P$ is smooth and thus $P \in \Pi_0$.
   Now, assume $k \geq 2$. Then \eqref{Lem.Pi_non-empty.Assertion1} implies that $X$ is an integral scheme.
   %Since $\dim X_{\sing} \leq n-k$, a sufficiently general $P \in \Gr$
   %does not intersect $X_{\sing}$. 
   Thus we may apply Bertini's theorem (see \cite[Corollary in \S5]{Kl1974The-transversality}): for a sufficiently general hyperplane $H \subseteq \PP^{n+1}$, the
   singular locus of $H \cap X$ is contained in $H \cap X_{\sing}$ and the latter
   has dimension $\leq n-k-1$. Proceeding by induction we find
   $P \in \Gr$ such that $P \cap X$ is smooth, i.e.\ $\Pi_{0}$ is non-empty.

   On the other hand, if $\dim X_{\sing} \geq n+1-k$, then 
   every $P \in \Gr$ intersects $X_{\sing}$. Using \Cref{Lem.X_cap_P_singular} we conclude
   that either $P$ is contained in $X$ or $X \cap P$ is singular. This shows that $\Pi_{0}$ is empty
   in this case.
\end{proof}

From now on, we will always assume that $\Pi_0$ is non-empty, i.e.
\begin{center}
   \begin{tcolorbox}[
  colback=white,
  colframe=black,
  arc=3mm,
  width = 4cm
   ]
   \begin{center}
   $\dim X_{\sing} + k \leq n$.
   \end{center}
   \end{tcolorbox}
\end{center}
For $k = 2$ (resp. $k = 1$), this means that $X$ is normal (resp. reduced).

\subsection{Preliminaries on the moduli of hypersurfaces.}
\label{Subsection.mpduli}
Denote by $\SS$ the set of those non-zero homogeneous  
degree $d$ polynomials in $x_0, \ldots, x_k$ that represent degree-$d$ hypersurfaces in $\PP^k$ that, with respect to the natural $\SL_{k+1}$-action, are semi-stable in the sense of geometric invariant theory.
Denote by $\SS_0 \subseteq \SS$ the homogeneous polynomials that give smooth 
hypersurfaces.
Then $\SS_0$ and $\SS$ are both open and dense in $\CC[x_0,\ldots,x_k]_d$, and the image $\PP \SS_0 \subseteq \PP \SS$ is contained
in the set of stable points, since $d > 2$ (\cite[\S10.1]{Do2003Lectures-on-invari}).
Let $\overline{\eta} \colon \PP \SS \to \PP \cquot{\SS}{\SL_{k+1}}$ 
be the categorical quotient with respect to $\SL_{k+1}$, where
$\PP \SS$ is the image of $\SS$ in $\PP(\kk[x_0, \ldots, x_k]_d)$. Then 
\[
   \eta \colon \SS \to \PP\SS \xlongrightarrow{\bar{\eta}} \cquot{\PP\SS}{\SL_{k+1}} 
\]
is the categorical quotient with respect to $\GL_{k+1}$,
so $\cquot{\PP\SS}{\SL_{k+1}} = \cquot{\SS}{\GL_{k+1}}$.

If $(d, k) = (3, 2)$, then the follwing holds: 
$\SS$ consists of those plane cubics with at worst nodal 
singularities, $\cquot{\SS}{\SL_{3}} = \PP^1$, $\eta(h) = [S^3(h) : S^3(h) + 2^6T^2(h)]$, 
where $S, T$ are homogeneous polynomials in the 
coefficients of $h \in \kk[x_0, x_1, x_2]_3$ of degree $4$ and  $6$, respectively,
and $\eta(\SS_0) = \AA^1 = \PP^1 \setminus \{ [1:0] \}$,
see \cite[\S 10.3]{Do2003Lectures-on-invari}.

If $(d, k) = (4, 1)$, then: 
$\SS$ consists of those binary quartics with at worst 
double roots, $\cquot{\SS}{\SL_{2}} = \PP^1$, $\eta(h) = [D^3(h) :  D^3(h) - 27E^2(h)]$, 
where $E, D$ are homogeneous polynomials in the 
coefficients of $h \in \kk[x_0, x_1]_4$ of degree  $3$ and $2$, respectively,
and $\eta(\SS_0) = \AA^1 = \PP^1 \setminus \{ [1:0] \}$,
see \cite[\S 10.2]{Do2003Lectures-on-invari}.

\subsection{Formulation of the main result}
Denote by $\Pi \subseteq \Gr$ the subset of those $P \in \Gr$ such that 
$P \not \subseteq X$ and $X \cap P \in \SS$. Hence we get a well-defined map 
\[
   \pi \colon \Pi \to \cquot{\PP \SS}{\SL_{k+1}}
   %\, , \quad\quad P \mapsto j(X \cap P) 
   \quad\quad
   \textrm{and we denote} \quad \SetC_X \coloneqq \pi(\Pi) \subseteq 
   \cquot{\PP \SS}{\SL_{k+1}}
\]
where $\pi(P) \coloneqq \eta(h)$ and  $h \in \kk[x_0, \ldots, x_k]_d$ is a polynomial 
with $X \cap P \simeq V(h) \subseteq \PP^k$.

\begin{lem}
   \label{Lem.Properties_of_Pi}
   The subsets $\Pi \subseteq \Pi_0$ are both open and dense in $\Gr$ and 
   $\pi \colon \Pi \to \cquot{\PP \SS}{\SL_{k+1}}$ is a morphism.
   % Moreover, $\Gr \setminus \Pi_{0}$ contains every $P \in \Gr$ such that 
   % $P \subseteq T_x X$ for some $x \in X \cap P$.
\end{lem}

\begin{proof}
   By our standing assumption, $\Pi_0$ is non-empty.

   Let $P \in \Gr$. We may choose the coordinates of $\PP^{n+1}$ in such a way that
   $P$ is given by the vanishing of  $x_{k+1}, x_{k+2}, \dots, x_{n+1}$. Then
   \[
      \iota \colon \Mat_{(n+1-k) \times (k+1)}(\kk) \to \Gr \, , \quad M \mapsto 
      V(M | E_{n+1-k})
   \]
   is an open embedding, where $V(M | E_{n+1- k })$ denotes the linear subspace in $\PP^{n+1}$
   cut out by the rows of $(M | E_{n+1-k}) \in \Mat_{(n+1-k) \times (n+2)}(\kk)$ and
   $E_{n+1-k}$ denotes the identity matrix of size $n+1-k$. 
   Denote the image of $\iota$ by $U \subseteq \Gr$. Then $P$ is the image of the zero matrix, and hence lies in $U$. It suffices
   to show that $U \cap \Pi$ and $U \cap \Pi_0$ are open in $U$ and that the restriction of $\pi$
   to $U \cap \Pi$ is a morphism. 
   
   Consider the morphism
   \[
      \rho \colon \Mat_{(n+1-k) \times (k+1)}(\kk) \to \kk[x_0, \ldots, x_k]_d \, , \quad  
      M \mapsto f(x_0, \ldots , x_k, -a_1, \dots, -a_{n+1-k} )
   \]
   where for $i \in \{1, \dots, n+1-k\}$ we define
   \[
      a_i = e_i M \begin{psmallmatrix} x_0 \\ \vdots\strut \\ x_k\end{psmallmatrix}
      \in \kk[x_0, \ldots, x_k]_1
   \]
   and $e_i$ denotes the $i$-th standard basis (row) vector. 
   Since $\SS$ is open, the preimage $\rho^{-1}(\SS)$ is an open subset
   of $\Mat_{(n+1-k) \times (k+1)}(\kk)$ that represents $U \cap \Pi$ inside $U$.
   Similarly we see that $U \cap \Pi_0$ is open in $U$.
   Moreover, the restriction of $\pi$ to $U \cap \Pi$ is represented by 
   the morphism
   \[
      \rho^{-1}(\SS) \xrightarrow{\rho} \SS \xrightarrow{\eta} 
      \cquot{\PP \SS}{\SL_{k+1}} \, ,
   \] 
   where $\eta$ is defined in \Cref{Subsection.mpduli}. 
   This implies the claim. \qedhere
   % 
   %The second statement is a direct consequence of \Cref{Lem.X_cap_P_singular}.
\end{proof}

We consider the set
\[
   L_X \coloneqq \set{x \in X}{\textrm{$X$ has multiplicity $d$ at $x$}} \subseteq X \, .
\]

\begin{re}
   \label{Rem.L_x_linear}
   The subset $L_X \subseteq X$ is a linear subspace of $\PP^{n+1}$. Indeed, if
   $p, q \in L_X$ are distinct points, then we may assume after a
   linear coordinate change of $\PP^{n+1}$
   that $p = [1:0:0: \cdots: 0]$, $q = [0:1:0: \cdots: 0]$. Hence,
   the equation that defines $X$ is a polynomial in the last $n$ variables
   $x_2, \dots, x_{n+1}$. This shows that
   the line joining $p,q$ is contained in $L_X$.
\end{re}

\begin{re} \label{re:LX}
    Let $m$ be the codimension of $L_X$ in $\PP^{n+1}$ (and $n+2$ if $L_X$ is empty). Then a suitable invertible linear change of coordinates maps $f$ into $\CC[x_0,\ldots,x_{m-1}]$. Thus the condition from \Cref{thm:Cubics} that $f$ uses at least $k+2$ 
    variables translates into the geometric condition that $L_X$ has 
    codimension at least $k+2$ in $\PP^{n+1}$, i.e., that $L_X$ has dimension $<n-k$.
\end{re}

\begin{re}
   \label{Rem.Locally-tribial-bundle}
   Using the \Cref{Rem.L_x_linear}, we can consider the projection from $L_X$. This yields a morphism
   \[
      p_X \colon \PP^{n+1} \setminus L_X \to \PP^{n - \dim L_X} \, .
   \]
   The image $p_{X}(X \setminus L_X)$ is closed in $\PP^{n - \dim L_X}$.
   Moreover, the restriction of $p_X$ to $X \setminus L_X$ gives an $\AA^{1+\dim L_X}$-bundle
   $X \setminus L_X \to p_{X}(X \setminus L_X)$, locally trivial with respect to the
   Zariski topology. In particular, $\dim p_{X}(X \setminus L_X) = n-\dim L_X-1$.
\end{re}

\begin{re}
   \label{Rem.Bound_dim_L_X}
   The hypersurface $X$ is singular along $L_X$.
   Hence, by our standing assumption: $\dim L_X \leq \dim X_{\sing} \leq n-k$.
\end{re}

The following is the main result of this section:

\begin{thm}\label{thm:n_gen}
Let $n \geq k$. We assume that $(d, k)$ is either $(3, 2)$ or $(4, 1)$ and hence
$\cquot{\PP \SS}{\SL_{k+1}} = \PP^1$.
If $\dim L_X < n-k$, then $\SetC_X = \PP^1$.
Otherwise, $\dim L_X = n-k$ and $\SetC_X$ is one point.
\end{thm}

Note that this immediately implies \Cref{thm:Cubics} by \Cref{re:LX} and the fact that the restriction of $\overline{\eta}$ to $\PP \SS_0$ is a geometric quotient and in particular yields a bijection between isomorphism classes of smooth cubic curves in $\PP^2$ 
(resp. isomorphism classes of sets of $4$ points in $\PP^1$) and $\AA^1$.

\smallskip

To prove \Cref{thm:n_gen}, we start with the following observation:

\begin{lem}
    \label{Lem.L_X_non-empty} If $\dim L_X = n-k$, then $\SetC_X$ is a point.
\end{lem}

\begin{proof}
Let $P \in \Gr$ such that $X \cap P$ is smooth (it exists as $\Pi_0$ is non-epmty). 
Then
$X_{\sing} \cap P$ is empty and hence so is $L_X \cap P$.
Therefore, $p_X \colon \PP^{n+1} \setminus L_X \to \PP^k$
restricts to an isomorphism $P \xrightarrow{\sim} \PP^k$. The latter restricts to an isomorphism
\[
   X \cap P \xrightarrow{\sim} p_X(X \setminus L_X) \, ,
\]
since $p_X(X \setminus L_X)$ is a closed irreducible hypersurface in $\PP^k$.
This shows that $\SetC_X$ consists only of one point, namely 
$\bar{\eta}(p_X(X \setminus L_X)) \in \cquot{\PP \SS}{\SL_{k+1}}$. 
\end{proof}

\begin{proof}[Proof of the second statement of \Cref{thm:n_gen}]
If $\dim L_X \geq n-k$, then $\dim L_X = n-k$ by \Cref{Rem.Bound_dim_L_X}
and the second statement follows from \Cref{Lem.L_X_non-empty}.
\end{proof}

We prove the first statement of \Cref{thm:n_gen} by induction on $n \geq 2$ and distinguish between
$(n, d, k) = (2, 3, 2)$ (\Cref{Prop.n_2_L_X_non-empty}), $(n, d, k) = (1, 4, 1)$ (\Cref{Prop.Quartics}) 
and $n > k$ 
(see \Cref{Section.Strictly_greater_than_2}). For proving the surjectivity of $\pi$
we use the following observation:

\begin{lem}
    \label{Lem.Technical_varphi_surj}
    Let $\varphi \colon \PP^{n+1} \dashrightarrow \PP^1$ be a dominant rational function
    and let $U$ be an open dense subset in the domain of $\varphi$.
    Assume that every irreducible hypersurface $W \subseteq \PP^{n+1}$ contained in 
    $\PP^{n+1} \setminus U$ satisfies:
    \begin{quote}
        If $\varphi |_{W} \colon W \dashrightarrow \PP^1$ is constant, 
        then there is a hypersurface $W'$ in $\PP^{n+1}$ with $W' \cap U \neq \emptyset$ 
        and $\varphi(W') = \varphi(W)$.
    \end{quote}
    Under this assumption, $\varphi(U) = \PP^1$.
\end{lem}

Denote by $\dom(\varphi)$ the domain of $\varphi$.
As $\PP^{n+1} \setminus \dom(\varphi)$ 
is of codimension $\geq 2$, the restriction 
of $\varphi$ to every hypersurface $W$ in $\PP^{n+1}$ is a well-defined rational map 
and we denote $\varphi(W) = \varphi(W \cap \dom(\varphi))$.

\begin{proof}
    There are homogeneous polynomials  $F_0, F_1 \in \kk[a_0, \ldots, a_{n+1}]$ 
    of the same degree without a common factor such that
    $\varphi(a) = [F_0(a): F_1(a)]$.
    Let $[b_0: b_1] \in \PP^1$.
    The fiber of $\varphi |_{\dom(\varphi)}$ over
   $[b_0:b_1]$ is the intersection of $\dom(\varphi)$ with the 
   (non-empty) hypersurface $Z \subseteq \PP^{n+1}$ given by
   $b_1 F_0 - b_0 F_1$.
   Towards a contradiction, we assume that $Z$ is contained 
   $\PP^{n+1} \setminus U$ and we choose an irreducible 
   hypersurface $W \subseteq \PP^{n+1}$ contained in $Z$.  
   Since $\varphi(Z) = \{ [b_0:b_1] \}$, 
   the assumption in the lemma guarantees that there is a hypersurface $W'$ with 
   $W' \cap U \neq \emptyset$ and 
   $\varphi(W') = \varphi(W) = \{ [b_0:b_1] \}$. Hence, $[b_0: b_1] \in \varphi(U)$, contradiction.
\end{proof}

We will also use the dual variety $\hat{X} \subseteq \Gr(n+1,n+2)$ of $X$. 
By definition, $\hat{X}$ is the closure of the image
of the Gauss map $\gamma \colon X_{\sm} \to \Gr(n+1,n+2)$, where $\gamma(x) = T_x X$ for 
$x \in X_{\sm}$. Alternatively, $\hat{X}$ is the image of 
$I_X \subseteq X \times \Gr(n+1,n+2)$ under the projection $X \times \Gr(n+1,n+2) \to \Gr(n+1,n+2)$, where
\[
   I_X \coloneqq \overline{\set{(x, P) \in X_{\sm} \times \Gr(n+1,n+2)}{P = T_x X}} \subseteq 
   X \times \Gr(n+1,n+2) \, .
\]

\subsection{The case where \texorpdfstring{$(n, d, k) = (2, 3, 2)$}{(n, d, k) = (2, 3, 2)}}

The goal of this subsection is to prove the first statement of 
\Cref{thm:n_gen} in case $n = k = 2$ and $d = 3$.
Hence, $X$ is a cubic surface in $\PP^3$ that is normal (by our standing assumption) and we intersect $X$ with planes $P \in \Gr=\Gr(3,4)$ of $\PP^3$.
As $X$ is a normal irreducible surface, $P \not\subseteq X$ for every $P \in \Gr$.

\smallskip

First we characterize the emptyness of $L_X$:

\begin{lem}
   \label{Lem.Char_L_X_empty}
   The following statements are equivalent:
   \begin{enumerate}[label=$(\arabic*)$]
      \item \label{Lem.Char_L_X_empty1} $L_X$ is empty;
      \item \label{Lem.Char_L_X_empty2} $X$ contains finitely many lines in $\PP^3$;
      \item \label{Lem.Char_L_X_empty3}
            $\dim \hat{X} = 2$;
      \item \label{Lem.Char_L_X_empty4}  for sufficiently general $x \in X$,
            $X \cap T_x X$ is an integral curve with exactly one singularity, which moreover is a node;
      \item \label{Lem.Char_L_X_empty5}
            $\pi \colon \Pi \to \PP^1$ is non-constant.
   \end{enumerate}
   %If $L_X$ is empty, then $\pi \colon \Pi_{0} \to \AA^1$ is non-constant.
\end{lem}

\begin{proof}
   \ref{Lem.Char_L_X_empty1} $\Leftrightarrow$ \ref{Lem.Char_L_X_empty2}: $L_X$ is nonempty 
   if and only if $X$ is a cone over a smooth plane cubic curve. Hence, the equivalence
   follows from the classification of normal cubic surfaces, see e.g. \cite[\S3]{BrWa1979On-the-classificat} or \cite[Theorem~1]{Sa2010Automorphism-group}.

   \ref{Lem.Char_L_X_empty2} $\Leftrightarrow$ \ref{Lem.Char_L_X_empty3}: This is a
   consequence of \cite[Theorem~1.18]{Te2005Projective-duality}, since 
   \ref{Lem.Char_L_X_empty1} $\Leftrightarrow$ \ref{Lem.Char_L_X_empty2}.

   \ref{Lem.Char_L_X_empty3} $\Rightarrow$ \ref{Lem.Char_L_X_empty4}:
   Let $P \in \hat{X}$ such that
   \begin{enumerate}[label={\roman*)}]
      \item~\label{Property1} $\hat{X}$ is smooth at $P$
      %\item~\label{Property2} $X \cap P \subseteq X_{\sm}$
      \item~\label{Property3} there are only finitely many $x \in X$ with $(x, P) \in I_X$
      \item~\label{Property4} $X$ is smooth at the elements satisfying~\ref{Property3}.
   \end{enumerate}
   Let $x_1, \dots, x_k \in X$ be the elements satisfying~\ref{Property3}.
   By~\ref{Property4} $X$ is smooth at $x_i$ for all $i$ and thus 
   $T_{x_i} X = P$ for all $i$. We have
   \begin{equation}
      \label{Eq.Dimca}
      1 = \sum_{i=1}^k \mu(X \cap P, x_i) \, ,
   \end{equation}
   where $\mu(X \cap P, x_i)$ denotes the Milnor number of $X \cap P$ at $x_i$ 
   (see \cite{Di1986Milnor-numbers-and}). By \Cref{Lem.X_cap_P_singular} we have that
   $X \cap P$ is singular at $x_i$ and hence 
   $\mu(X \cap P, x_i) > 0$ for all $i$. Using~\eqref{Eq.Dimca} we get $k = 1$ 
   and $\mu(X \cap P, x_1) = 1$.
   
   \smallskip

   Now, if $x \in X_{\sm}$ such that $x$ does not belong to any line inside $X$,
   then 
   \begin{enumerate}[label={\alph*)}]
      \item \label{Property_curve1} $X \cap T_x X$ is an integral curve
      \item \label{Property_curve2} $x$ is the unique singularity of $X \cap T_x X$ and it is a node or an ordinary cusp
   \end{enumerate}
   by the classification of plane cubic curves, see e.g.\ \cite[\S2.2]{EiHa20163264-and-all-that-}.
   Since $X$ contains
   only finitely many lines 
   (by~\ref{Lem.Char_L_X_empty3} $\Rightarrow$ \ref{Lem.Char_L_X_empty2}), we have that 
   \ref{Property_curve1} and \ref{Property_curve2} are satisfied for sufficiently general $x \in X$.
   Since the Milnor number of a plane node is $1$ and for a plane ordinary cusp is $2$, 
   by the proceeding paragraph 
   it is enough to show that sufficiently general $P \in \hat{X}$ satisfy \ref{Property1}, \ref{Property3} and~\ref{Property4}.

   As $I_X$ is the closure in $X \times \Gr$ of the graph of the Gauss map $\gamma \colon X_{\sm} \to \Gr$, the projection $q \colon I_X \to X$ is dominant. 
   Hence $q^{-1}(X_{\sing})$ has dimension $\leq 1$ and thus $Z \coloneqq p(q^{-1}(X_{\sing}))$
   is a proper closed subset of $\hat{X}$ (since $\dim \hat{X} = 2$). Thus for 
   $P \in \hat{X} \setminus Z$ the following holds: If $(x, P) \in I_X$, then $x \in X_{\sm}$. 
   Now take a dense open subset $U \subseteq (\hat{X})_{\sm} \setminus Z$ such that 
   $\gamma \colon X_{\sm} \to \Gr$ has finite fibers over $U$. Then every $P \in U$
   satisfies~\ref{Property1}, \ref{Property3} and~\ref{Property4}.

   \ref{Lem.Char_L_X_empty4} $\Rightarrow$ \ref{Lem.Char_L_X_empty5}: For sufficiently general $x \in X$
   the singular nodal cubic curve $X \cap T_x X$ is sent onto the point at infinity via $\pi$.
   As $\pi(\Pi_0) \subseteq \AA^1$ is non-empty, $\pi \colon \Pi \to \PP^1$ is non-constant. 

   \ref{Lem.Char_L_X_empty5} $\Rightarrow$ \ref{Lem.Char_L_X_empty1}: This follows
   from \Cref{Lem.L_X_non-empty}.
\end{proof}

\begin{lem}
   \label{Lem.Ak-singularity}
   Assume $L_X$ is empty and $X$ is singular at $x$ with reduced tangent cone. Then 
   for sufficiently general $P \in \Gr$ with $x \in P$, the curve $X \cap P$ has a node at $x$. 
\end{lem}

\begin{proof}
   After changing the coordinates, we may assume $x = [1:0:0:0]$ and since $X$
   is singular at $x$, the hypersurface $X$ is given by
   \[
      f = x_0 f_2(x_1, x_2, x_3) + f_3(x_1, x_2, x_3) \, ,
   \]
   where $f_i$ is homogeneous of degree $i$. 
   Note that $f_2 \neq 0$, since $L_X$ is empty.
   Up to a linear coordinate change in 
   $x_1, x_2, x_3$ we may assume that $f_2 \in \{x_1^2, x_1 x_2, x_1 x_3 - x_2^2\}$.
   Since the tangent cone of $X$ at $x$ is reduced, we have $f_2 \neq x_1^2$.

   \textbf{Case 1: $f_2 = x_1 x_3-x_2^2$.} In the chart $\{ x_0 \neq 0 \}$, the intersection of
   $X \cap P$ is given by
   \begin{equation}
      \label{Eq.Node1}
      (a_2 x_2 + a_3 x_3) x_3 - x_2^2 + \textrm{terms of order $3$} \, ,
   \end{equation}
   where $P$ is given by $x_1 - a_2 x_2 - a_3 x_3$. If the discriminant $a_2^2+4a_3 \neq 0$, 
   then $(a_2 x_2 + a_3 x_3) x_3-x_2^2$ is not a square and 
   thus~\eqref{Eq.Node1} defines a node at $(0, 0, 0)$.

   \textbf{Case 2: $f = x_1 x_2$.} In the chart $\{ x_0 \neq 0 \}$, the intersection of
   $X \cap P$ is given by
   \begin{equation}
      \label{Eq.Node2}
      (a_2 x_2 + a_3 x_3) x_2 + \textrm{terms of order $3$} \, ,
   \end{equation}
   where $P$ is given by $x_1 - a_2 x_2 - a_3 x_3$. Now $X \cap P$
   has a node at $x$ as long as $a_3 \neq 0$.
\end{proof}

\begin{re}
   \label{Rem.Nodes}
   If $P \cap X$ has a node, then all singularities of $P \cap X$ are nodes 
   by the classification of plane cubics 
   (see~\cite[\S2.2]{EiHa20163264-and-all-that-}).
\end{re}

\begin{prop}
   \label{Prop.n_2_L_X_non-empty}
   If $L_X$ is empty, then $\SetC_X = \PP^1$.
\end{prop}

\begin{proof}
   The morphism $\pi$ represents a rational map 
   $\varphi \colon \Gr \simeq \PP^3 \dashrightarrow \PP^1$.
   Due to \Cref{Lem.Technical_varphi_surj} 
   it is enough to show that the restriction of $\varphi$ to 
   every closed irreducible hypersurface inside $\Gr \setminus \Pi$ is non-constant.
   
   Denote by $q_1, \dots, q_r$ the singularities of $X$. Then $\Gr \setminus \Pi_0$ is the union of the dual variety $\hat{X}$ and the 
   hyperplanes $\mathcal{H}_i \coloneq \set{P \in \Gr}{q_i \in P}$ 
   for $i = 1, \dots, r$. Hence,
   \[ \Gr \setminus \Pi = (\hat{X} \setminus \Pi) \cup \bigcup_{i=1}^r (\mathcal{H}_i \setminus \Pi).\]
   Note that, since $L_X$ is empty, $\hat{X}$ is an irreducible surface (\Cref{Lem.Char_L_X_empty}) and so are all $\mathcal{H}_i$.
   Since $\Pi \cap \hat{X} \neq \emptyset$ (see \Cref{Lem.Char_L_X_empty}),
   we have $\dim(\hat{X} \setminus \Pi) \leq 1$. 
   Furthermore, if the tangent cone of $X$ at $q_i$ is reduced, then $\mathcal{H}_i \cap \Pi$
   is non-empty due to \Cref{Lem.Ak-singularity} and \Cref{Rem.Nodes}
   and $\dim \mathcal{H}_i \setminus \Pi \leq 1$.
    
   Hence, every irreducible hypersurface in $\Gr \setminus \Pi$ coincides with some $\mathcal{H}_i$, corresponding to a point $q_i$ at which the tangent cone of $X$ is non-reduced. In particular, we may assume that the tangent cone of $X$ at a singular point, 
   say $q = [1:0:0:0]$, is non-reduced.
   After a linear change of coordinates, we have
   $f = x_0 x_1^2 + f_3(x_1, x_2, x_3)$
   for a suitable homogeneous polynomial $f_3$ of degree $3$.
   By the proof of 
   \cite[Lemma~4]{BrWa1979On-the-classificat}, 
   up to a linear coordinate change, $f$ is one of the following polynomials:
   \begin{enumerate}[label={\alph*)}]
      \item \label{Case1} $x_0 x_1^2 + x_2^3 + x_3^3$
      \item \label{Case2} $x_0 x_1^2 + x_2^3 + x_3^3 + x_1 x_2 x_3$
      \item \label{Case3} $x_0 x_1^2 + x_2^2 x_3 + x_1 x_3^2$
      \item \label{Case4} $x_0 x_1^2 + x_2^3 + x_1 x_3^2$ 
   \end{enumerate}
   By the classification of normal cubic surfaces, in each case 
   $q$ is the unique singularity 
   (see \cite[\S3]{BrWa1979On-the-classificat} or \cite[Theorem~1]{Sa2010Automorphism-group}). 
   In order to conclude the proof, 
   in Cases~\ref{Case1}, \ref{Case2}, we will show that $\varphi |_{\mathcal{H}}$
   is non-constant for $\mathcal{H} \coloneqq \set{P \in \Gr}{q \in P}$ and
   in Cases~\ref{Case3}, \ref{Case4}, we will prove the surjectivity
   of $\pi = \varphi |_{\Pi}$ directly.

   \textbf{Cases~\ref{Case1}, \ref{Case2}}: 
   In order to treat these cases simultaneously, 
   we consider 
   the polynomial $x_0 x_1^2 + x_2^3 + x_3^3 + \lambda x_1 x_2 x_3$, where $\lambda \in \kk$. Now $\varphi$ is given by
   \[
      \varphi \colon \PP^3 \dashrightarrow \PP^1 \, , \quad
      [a_0: a_1: a_2 : a_3] \mapsto 
      [S^3(g) : 2^6 S^3(g) + T^2(g)] \, 
   \]
   where $g = (a_1 x_1 + a_2 x_2 + a_3 x_3) x_1^2 - a_0(x_2^3 + x_3^3 + \lambda x_1 x_2 x_3)$
   and $S, T$ are introduced in \Cref{Subsection.mpduli}. 
   The elements of $\mathcal{H}$ correspond to the tuples 
   $[0: a_1 : a_2 : a_3]$. The restriction of $\varphi$
   to the hyperplane $\mathcal{H}$ is then given by the rational map
   \[
      \PP^2 \dashrightarrow \PP^1 \, , \quad [a_1: a_2 : a_3] 
      \mapsto [-(a_2a_3)^3: 2^4 (a_2^3 - a_3^3)^2] \, ;
   \]
   see the computation using SageMath from \Cref{SageMathCode}. 
   In particular, this restriction is not constant, so $\mathcal{H}$ is not contained in the single fiber $Z$.

   \textbf{Case~\ref{Case3}}: In this case we will show by explicit computations that intersecting $X$ with suitable planes we obtain enough curves for $\pi$ to hit all of $\PP^1$. Let $a_1, a_3 \in \kk$. Then the intersection 
   of the hyperplane given by $x_0 + a_1 x_1 + a_3 x_3=0$ with $X$ is isomorphic
   to the curve in $\PP^2$ given by
   \[
      x_2^2 x_3 = a_1 x_1^3 + a_3 x_1^2 x_3 -x_1 x_3^2 \, .
   \]
   If $a_1 = 0$, the the above curve is reducible, hence not smooth. We assume in the following
   that $a_1 \neq 0$. We multiply $x_2$ by $\sqrt{a_1}$ and then arrive at the curve given by
   \[
      x_2^2 x_3 = x_1^3 + a_3 a_1^{-1} x_1^2 x_3 - a_1^{-1} x_1 x_3^2 \, .
   \]
   Let $\lambda \coloneqq \frac{a_3 a_1^{-1}}{3}$. After replacing $x_1$ by $x_1 - \lambda x_3$,
   we get the curve in $\PP^2$ given by
   \[
      x_2^2 x_3 = x_1^3 + (-3\lambda^2-a_1^{-1}) x_1 x_3^2 + (2 \lambda^3 + a_1^{-1} \lambda) x_3^3 \, .
   \]
   We claim that the image of
   \[
      f \colon (\AA^1 \setminus \{0\}) \times \AA^1 \to \AA^2 \, , \quad
      (c, \nu) \mapsto (-3\nu^2-c, 2 \nu^3 + c \nu)
   \]
   contains the complement of the curve 
   $C \coloneqq \set{(-3 \nu^2, 2 \nu^3)}{\nu \in \kk}$ in $\AA^2$.
   Indeed, note that $f$ is the restriction to $(\AA^1 \setminus \{0\}) \times \AA^1$ of
   the following composition:
   \[
         F \colon \AA^2 \xrightarrow{ (c, \nu)  \mapsto (3\nu^2+c, \nu)} 
         \AA^2 \xrightarrow{(w, \nu) \mapsto (-w, \nu w-\nu^3)} \AA^2  \, .
   \]
   The first map of the composition above is an automorphism, 
   while the second one is surjective and hence $F$ is surjective. Since
   $F$ maps $\{0\} \times \AA^1$ onto $C$, the claim follows. Now 
   $x_2^2 x_3 = x_1^3 + a x_1 x_3^2 + b x_3^3$ is smooth if and only if the discriminant
   $4a^3+27b^2$ is non-zero, and this happens if and only if $(a, b)$ is not contained in the curve $C$. Moreover, all isomorphism classes of smooth cubic curves have such a Weierstrass form, so we conclude that $\pi(\Pi_0)=\AA^1$. Since $\pi$ maps the non-empty set $\hat{X} \cap \Pi$ to the point at infinity, we are done.

   \textbf{Case~\ref{Case4}}: Let $a_1, a_2 \in \kk$. Then the intersection 
   of the hyperplane given by $x_0 - a_1 x_1 - a_2 x_2$ with $X$ is isomorphic
   to the curve in $\PP^2$ given by $(a_1 x_1 + a_2 x_2) x_1^2 + x_2^3 + x_1 x_3^2$.
   % \[
   %    x_1 x_3^2 = -a_1 x_1^3 -a_2 x_1^2 x_2 - x_2^3 \, .
   % \]
   Multiplying $x_3$ with $\i \in \kk$ yields the curve in $\PP^2$ given by
   \[
      x_1 x_3^2 = a_1 x_1^3 + a_2 x_1^2 x_2 + x_2^3 \, .
   \]
   We see that all Weierstrass forms arise in this manner, so that $\pi$ is surjective as in Case~\ref{Case3}.
   %[TO BE FINISHED]
\end{proof}

\subsection{The case where \texorpdfstring{$(n, d, k) = (1, 4, 1)$}{(n, d, k) = (1, 4, 1)}}

The goal of this subsection is to prove the first statement of 
\Cref{thm:n_gen} in case $n = k = 1$ and $d = 4$.
Hence, $X$ is a quartic curve in $\PP^2$ that is reduced (by our standing assumption) and we intersect $X$ with lines $P \in \Gr=\Gr(2,3)$ in $\PP^2$.

\begin{prop}
   \label{Prop.Quartics}
   If $L_X$ is empty, then $\SetC_X = \PP^1$.
\end{prop}

For the proof we use the following lemma:

\begin{lem}[{see e.g.~\cite[Propositions~1.1.5 and 1.1.11]{Do2012Classical-algebrai}}]
   \label{Lem.Finitely_many_flex_points}
   Assume that $Z \subseteq \PP^2$ is an irreducible closed curve of degree $> 1$.
   Then, there exist only finitely many $z \in Z$ such that the intersection multiplicity 
   of $Z$ and $T_z Z$ at $z$ is $\geq 3$. \qed
\end{lem}

\begin{proof}[Proof of \Cref{Prop.Quartics}]
   The morphism $\pi \colon \Pi \to \PP^1$
   represents a rational map $\varphi \colon \PP^2 \simeq \Gr \dashrightarrow \PP^1$ 
   which is defined everywhere, except in a finite number of points. 
   Due to \Cref{Lem.Technical_varphi_surj} 
   it is enough to show the following for all irreducible curves 
   $\mathcal{W}$ contained in $\Gr \setminus \Pi$:

   \begin{quote}
        If $\varphi|_{\mathcal{W}}$ is constant, then there is a curve $\mathcal{W}'$ 
        that intersects $\Pi$ and $\varphi(\mathcal{W}') = \varphi(\mathcal{W})$.
   \end{quote}

   Let $q_1, \ldots, q_r \in X$ be the singularities of $X$.
   Denote by $\mathcal{H}_i 
   \subseteq \Gr$ the hyperplane of lines though $q_i$. Then 
   \[
      \Gr \setminus \Pi = (\hat{X} \setminus \Pi) \cup \bigcup_{i=1}^r 
      (\mathcal{H}_i \setminus \Pi) \, .
   \] 
   
   We claim that $\hat{X} \setminus \Pi$ is finite, $\varphi$ is dominant 
   and there is a curve $\mathcal{W}_{\infty} \subseteq \Gr$ with
   $\mathcal{W}_{\infty} \cap \Pi \neq \emptyset$ and 
   $\varphi(\mathcal{W}_{\infty}) = \{[1:0]\}$. Indeed,
   due to \Cref{Lem.Finitely_many_flex_points} applied to the irreducible 
   components $Z$ of $X$ that are not lines, we have that 
   $T_z Z \in \Pi$ for all but finitely many smooth points $z \in Z$. 
   Hence, $\hat{X} \setminus \Pi$ is finite. 
   If $X$ consists only of lines, then we choose $\mathcal{W}_{\infty}$
   to be the set of lines through one singularity of $X$ of multiplicity $2$. Such a singularity exists: since $L_X$ is empty, not all four lines go through a common point, and it follows that there is a point where precisely two of the lines meet. We get $\mathcal{W}_{\infty} \cap \Pi \neq \emptyset$,
   $\varphi(\mathcal{W}_{\infty}) = \{[1:0]\}$. If $X$ contains an irreducible
   component $Y$ of degree $> 1$, then $\dim \hat{X} = 1$ and choosing 
   $\mathcal{W}_{\infty} = \hat{Y}$ yields $\mathcal{W}_{\infty} \cap \Pi \neq \emptyset$, 
   $\varphi(\mathcal{W}_{\infty}) = \{[1:0]\}$.
   Moreover, $P \cap X$ is reduced for some $P \in \Gr$ (as $X$ is reduced)
   and thus $\pi(P) \in \AA^1$. Thus, $\varphi$ is dominant.

   If $X$ has multiplicity $2$ at $q_i$, then for sufficiently general 
   $P \in \mathcal{H}_i$ we have that the multiplicity of $P \cap X$ at $q_i$ is $2$ and thus $\mathcal{H}_i \setminus \Pi$ is finite.

   Now, assume $X$ has multiplicity $3$ at some point, say at $q \coloneqq q_1$.
   Choose coordinates such that $q = [1:0:0]$ and then $\mathcal{H} \coloneqq \mathcal{H}_i$
   consists of the lines given by $a x_1 + b x_2=0$ for $(a, b) \neq (0, 0)$. 
   Note that $\mathcal{H} \subseteq \Gr \setminus \Pi$. It suffices to show
   that if $\varphi |_{\mathcal{H}}$  is constant, then there is an irreducible curve in $\Gr$ that intersects $\Pi$ and maps 
   to the same point as $\mathcal{H}$ via $\varphi$.
   
   There are $f_j \in \kk[x_1, x_2]_j$ for $j = 3, 4$ with
   \[
      f = x_0 f_3(x_1, x_2) + f_4(x_1, x_2) \, ,
   \] 
   and $f_3$ is non-zero. Up to an invertible linear coordinate change in $x_1, x_2$, we have
   \[
      f_3 \in \{ x_1 x_2 (x_1  + x_2), x_1^2 x_2, x_1^3 \} \, .
   \]

   \textbf{Case 1: $f_3 \in \{ x_1 x_2 (x_1 + x_2), x_1^2x_2 \}$:} 
   Take $a \in \kk \setminus \{0, 1\}$ such that $\varphi$ is defined at the line given by 
   $a x_1 + x_2$
   (which belongs to $\mathcal{H}$). For $t \in \kk$ denote by 
   $L_t^a$ the line given by $t x_0 - (a x_1 + x_2)$. Then, for $t \neq 0$ the intersection 
   $X \cap L_t^a$ inside $L_t^a$ is given by
   \[
      (a x_1 + x_2)f_3(x_1, x_2) + t f_4(x_1, x_2) \, .
   \]
   Since the above equation has at most double roots for $t = 0$, 
   the same holds for sufficiently general $t \in \kk$. 
   % Consider the morphism
   % \[
   %    \theta \colon \AA^1 \to \PP^1 \, , \quad t \mapsto
   %    \eta\left( (a x_0 + x_1)x_0 x_1 (x_0+ x_1) + t f_4(x_0, x_1) \right)
   % \]
   % (this rational map is in fact a morphism, since $\AA^1$ is a smooth curve).
   Then $\varphi(L_t^a) = \pi(L_t^a)$ for sufficently general $t \in \kk$ and
   \begin{align*}
      \varphi(L_0^a) = \lim_{t \to 0} \varphi(L_t^a) &= 
      \lim_{t \to 0} \eta \left( (a x_1 + x_2)f_3(x_1, x_2) + t f_4(x_1, x_2) \right) \\
      &= \eta \left( (a x_1 + x_2)f_3(x_1, x_2) \right) \, .
   \end{align*}
   This implies that $\varphi |_{\mathcal{H}}$ is non-constant
   in case $f_3 = x_1 x_2 (x_1 + x_2)$. In case $f_3 = x_1^2 x_2$,
   we see that $\varphi(\mathcal{H}) = \{[1:0] \} \in \PP^1$ and we use $\varphi(\mathcal{W}_{\infty}) = \{[1:0]\}$.

   \textbf{Case 2: $f_3 = x_1^3$:} Write $f_4 = \sum_{i=0}^4 \lambda_i x_1^{4-i} x_2^i$.
   If $\lambda_4 \neq 0$, we may assume after replacing $x_2$ by 
   $x_2 - \frac{\lambda_3}{4\lambda_4} x_1$ that $\lambda_3 = 0$.
   After scaling $x_0, x_1, x_2$ we may further assume that
   $(\lambda_2, \lambda_4)$ is equal to $(1, 1)$ or to $(0, 1)$. Moreover, 
   after replacing $x_0$ with $x_0 - \lambda_0 x_1 - \lambda_1 x_2$, we may in addition assume
   that $\lambda_0 = \lambda_1 = 0$.

   If $\lambda_4 = 0$, then $\lambda_3 \neq 0$, since otherwise
   $x_1^2$ would divide $f$. Similarly as before,
   we may assume $\lambda_0 = \lambda_1 = \lambda_2 = 0$ and $\lambda_3 = 1$.

   Hence, we are reduced to study the cases 
   $f_4 \in \{ x_1^2 x_2^2 + x_2^4, x_2^4, x_1 x_2^3  \}$. 
   Calculating the gradient of $f$, 
   we see that $q$ is the only singularity of $X$ and hence
   $\mathcal{H}$ is the only irreducible curve inside $\Gr \setminus \Pi$.
   Let us define a curve $\mathcal{K} \subseteq \Gr$ with respect to $f_4$ as follows
   (the coordinates in $\Gr \simeq \PP^2$ are given by $[a_0: a_1:a_2]$):
   \begin{center}
      \begin{tabular}{r|c|c|c}
         $f_4$ & $x_1^2 x_2^2 + x_2^4$ & $x_2^4$ & $x_1 x_2^3$ \\
         \hline
         $\mathcal{K}$ & $a_0 - 12a_1$ & $a_1$ & $a_1$
      \end{tabular} \, .
   \end{center}
   Hence, $\mathcal{K} \neq \mathcal{H}$.
   A SageMath computation (see \Cref{SageMathCode}) shows
   $\varphi(\mathcal{H}) = \varphi(\mathcal{K})$.

   \smallskip
   
   The analysis above implies now, that $\pi \colon \Pi \to \PP^1$ is surjective.
\end{proof}

\subsection{The case where \texorpdfstring{$n > k$}{n > k}}
\label{Section.Strictly_greater_than_2}
Now we can prove the first statement of \Cref{thm:n_gen} for $n > k$
by induction. The crucial point is the following observation:

\begin{lem}
\label{Lem.reduction} Let $n \geq k+1$.
Assume that $\dim L_{X \cap H} = n - 1-k$ for every hyperplane $H$ such that 
$\dim (X \cap H)_{\sing} + k \leq n-1$. Then $\dim L_{X} = n-k$.
\end{lem}

\begin{proof}
Consider
\[
   \mathcal{H} \coloneqq \Bigset{H}{
      \begin{array}{l}
         \textrm{$H$ is a hyperplane such that} \\
         \textrm{$\dim (H \cap X)_{\sing} + k \leq n-1$ and $H \not \in \hat{X}$}
      \end{array}
   } \, .
\]
Then $\mathcal{H}$ is open and dense in the space of all hyperplanes, since
$\dim X_{\sing} + k \leq n$ and 
$(X \cap H)_{\sing} =  X_{\sing} \cap H$
for $H \in \mathcal{H}$
(here we use that $H \not \in \hat{X}$).

Recall that $p_{X \cap H} \colon (X \cap H) \setminus L_{X \cap H} \to Z_X$ is a locally
trivial $\AA^{n-k}$-bundle, where $Z_X \subset \PP^k$ is closed and irreducible, and we see $X \cap H$
as a hypersurface in $H \simeq \PP^n$ (\Cref{Rem.Locally-tribial-bundle}). 
Since $\dim (X \cap H)_{\sing} \leq n-1-k$
we have that $Z_X$ is smooth and thus $(X \cap H) \setminus L_{X \cap H}$
is smooth. This shows that
\[
   \label{Eq.multiplicity3_eq_sing}
   X_{\sing} \cap H = (X \cap H)_{\sing} = L_{X \cap H} \quad\quad
   \textrm{for all $H \in \mathcal{H}$} \, .
\]

We claim that $L_{X \cap H} \subseteq L_X$ for all $H \in \mathcal{H}$.
Indeed, let $x \in L_{X \cap H} = X_{\sing} \cap H$ for some $H \in \mathcal{H}$
and denote $\mathcal{H}_x \coloneqq \set{H' \in \mathcal{H}}{x \in H'}$. 
Then $x \in X_{\sing} \cap H' =  
L_{X \cap H'}$ for all $H' \in \mathcal{H}_x$
and since $\mathcal{H}_x$ is open and dense in the linear space
of all hyperplanes through $x$, we conclude that $X$ 
has multiplicity $d$ at $x$, i.e.\ $x \in L_X$.

Now, $\dim L_X \geq \dim L_{X \cap H} = n-1-k$ for all $H \in \mathcal{H}$. 
If $\dim L_X = n-1-k$, then $L_{X \cap H} = L_X$ for all $H \in \mathcal{H}$,
and thus $L_X$ is contained in every $H \in \mathcal{H}$, contradiction
(note that 
$L_X \neq \emptyset$, as $n \geq k+1$). Thus, $\dim L_X = n-k$
(by \Cref{Rem.Bound_dim_L_X}).
\end{proof}

\begin{proof}[Proof of \Cref{thm:n_gen} when $n > k$]
We assume $L_X < n - k$. Using \Cref{Lem.reduction}
we get a hyperplane $H$ such that $\dim (H \cap X)_{\sing} + k \leq n-1$ 
and $\dim L_{X \cap H} < n-1-k$.
By the induction hypothesis applied to $X \cap H$ in $H \simeq \PP^n$ we get $\SetC_{X \cap H} = \PP^1$. Since $\SetC_{X \cap H} \subseteq \SetC_X$, 
we conclude that $\SetC_X = \PP^1$.
\end{proof}

\subsection{\texorpdfstring{Proof of \Cref{thm:LowRank}}{Proof of Theorem 1.7}}

\begin{proof}[Proof of \Cref{thm:LowRank} using \Cref{thm:Cubics}.]
\begin{enumerate}
    \item 
Let $f \in S^3 V$ have $\underline{Q}_s(f) \leq 3$. If $\underline{Q}_s(f)=0$, then $f=0$ and also $Q_s(f)=0$. If $\underline{Q}_s(f)=1$, then $f \neq 0$ and for a sufficiently general linear map $\varphi:V \to \CC$ we have $\varphi f \neq 0$, so that $\varphi f$ is (a nonzero scalar multiple of) $e_1^3$. Hence $Q_s(f)=1$. If $\underline{Q}_s(f)=2$, then the locus $\{\varphi f \mid \varphi \in \Hom(V,\CC^2)\}$ has $I_2:=e_1^3 + e_2^3$ in its closure. But the $\GL_2$-orbit of this latter polynomial is dense in $S^3 \CC^2$, and hence that orbit intersects said locus. We conclude that $Q_s(f)=2$. 

It remains to study the most interesting case, namely, $\underline{Q}_s(f)=3$. If $\dim(V)=3$, then $f$ has maximal symmetric border subrank, and hence maximal symmetric subrank $3$ by \Cref{lm:max_sym_bordersubrank}. This also applies if $\dim(V) \geq 4$ and $f$ uses only $3$ variables. So we may assume that $\dim(V) \geq 4$ and $f$ uses at least $4$ variables. Let $X \subseteq \PP V^*$ be the cubic hypersurface defined by $f$. For a linear surjection $\varphi:V \to \CC^3$ that does not map $f$ to zero, $\varphi f$ is the equation for the projective curve in $\PP^2$ that is the (scheme-theoretic) pre-image of $X$ under the dual injection $\PP^2 \to \PP V^*$ induced by $\varphi$. The assumption $\underline{Q}_s(f) \geq 3$ means that the Fermat cubic $e_1^3+e_2^2+e_3^3$ lies in the closure of the locus of all such $\varphi f$. In particular, $X$ is normal, because otherwise all $\varphi f$ would define singular curves, and the Fermat cubic is smooth. Hence the conditions of \Cref{thm:Cubics} are satisfied, and we conclude that there exists a $\varphi$ such that $\varphi f$ equals the Fermat cubic. 

\item Now let $f \in S^4 V$ have $\underline{Q}_s(f) \leq 2$. If $\underline{Q}_s(f) \leq 1$, then we can argue as above to show that $Q_s(f)=\underline{Q}_s(f)$. So suppose that $\underline{Q}_s(f)=2$. For a linear surjection $\varphi:V \to \CC^2$ that does not map $f$ to zero, $\varphi f$ is the equation of the finite set of points on $\PP^1$ that is the pre-image of $X$ under the dual injection $\PP^1 \to \PP V^*$. The assumption that $\underline{Q}_s(f)=2$ implies that the equation $e_1^4+e_2^4$ lies in the closure of the locus of all such $\varphi f$. In particular, $X$ is reduced. Hence all conditions of \Cref{thm:Cubics} are satisfied, and we conclude that there exists a $\varphi$ such that $\varphi f=e_1^4+e_2^4$. \qedhere
\end{enumerate}

\end{proof}

\begin{re}
    The proof above shows how much stronger \Cref{thm:Cubics} is than \Cref{thm:LowRank}: in the latter part of the proof, not only the Fermat cubic, but {\em any} smooth cubic curve can be obtained from $f$ by a suitable $\varphi:V \to \CC^3$; 
    and this applies similarly to quartic hypersurfaces and (crossratios of) $4$ points in $\PP^1$.
\end{re}

\appendix
\section{SageMath codes}
\label{SageMathCode}

The following SageMath code is used in the proof of  
\Cref{Prop.n_2_L_X_non-empty}:

\begin{center}
\begin{minipage}{\textwidth} 
\begin{lstlisting}
SR.<lam, a0, a1, a2, a3> = QQ[]
P.<x1, x2, x3> = PolynomialRing(SR)

poly = (a1*x1 + a2*x2 + a3*x3)*x1^2 \
        - a0*(x2^3 + x3^3 + lam*x1*x2*x3)
cubic = invariant_theory.ternary_cubic(poly)
     
S = cubic.S_invariant()
T = cubic.T_invariant()
    
A = S^3
B = 2^6*S^3 + T^2
g = gcd(A, B)

A_ = (A // g).subs(a0 = 0)
B_ = (B // g).subs(a0 = 0)

show((A_, B_))
\end{lstlisting}
\end{minipage}
\end{center}

\bigskip
The following SageMath code is used in the proof of \Cref{Prop.Quartics}:

\begin{center}
\begin{minipage}{\textwidth} 
\begin{lstlisting}
SR.<a_0, a_1, a_2> = QQ[]
P.<x_1, x_2> = PolynomialRing(SR)

def normalized_pair(A, B):
    g = gcd(A, B)
    return A // g, B // g

def pi(poly):
    q = invariant_theory.binary_quartic(poly, x_1, x_2)
    E = q.EisensteinE()
    D = q.EisensteinD()
    return normalized_pair(D^3, D^3 - 27*E^2)

T = [
    (1, 1/12, a_2, x_1^2*x_2^2 + x_2^4),
    (1, 0,    a_2, x_2^4),
    (1, 0,    a_2, x_1*x_2^3),
]

for b0, b1, b2, term in T:
    A, B = pi((a_1*x_1 + a_2*x_2)*x_1^3 - a_0*term)
    show(normalized_pair(A.subs({a_0: 0}),
                         B.subs({a_0: 0})))

    show(normalized_pair(A.subs({a_0: b0, a_1: b1, a_2: b2}),
                         B.subs({a_0: b0, a_1: b1, a_2: b2})))
\end{lstlisting}
\end{minipage}
\end{center}

\bibliographystyle{alpha}\bibliography{math}

\end{document}